%% file: TP_new.tex
\documentclass[reqno]{llncs}

\usepackage[latin1]{inputenc}
\usepackage{pgf}

\usepackage{amssymb}
\usepackage{amsmath}
\usepackage{fancyhdr}
\usepackage{setspace}
\usepackage{tikz}
\usepackage{here}
\usepackage{paralist} 
\usetikzlibrary{snakes}

\usepackage{algorithm}

\newcommand{\vey}{y}
\newcommand{\veu}{u}
\newcommand{\vev}{v}
\newcommand{\si}{\sigma}
\newcommand{\di}{\delta}
\newcommand{\R}{\mathbb{R}}
\newcommand{\crit}{\mu}
\newcommand{\TP}{\text{TP}(u,v)}
\newcommand{\se}{e} % shaded edge
\newcommand{\sealg}{\se} % shaded edge
\newcommand{\suc}{C'} % shaded edge

\begin{document}

\title{The Diameters of  Network-flow Polytopes satisfy the Hirsch conjecture}
\author{S. Borgwardt\inst{1} \and J. A. De Loera\inst{2} \and E. Finhold\inst{3}}

\pagestyle{plain}

\institute{University of Colorado, Denver \and University of California, Davis \and Fraunhofer-Institut f\"ur Techno- und Wirtschaftsmathematik ITWM, Kaiserslautern
}

\maketitle

\begin{abstract}
\noindent We solve a problem in the combinatorics of polyhedra motivated by the network simplex method. We show 
that the Hirsch conjecture holds for the diameter of the graphs of all network-flow polytopes, in particular
the diameter of a network-flow polytope for a network with $n$ nodes and $m$ arcs is never more than 
$m+n-1$. A key step to prove this is to show the same result for classical transportation polytopes.

%\keywords{combinatorial diameter  \and Hirsch conjecture \and transportation polytopes }
%\subclass{52B05, 90C05, 90C08} TODO: CONTROL
\end{abstract}


\input{TP_Intro}

\section{Algorithm to prove Theorem \ref{main thm}}\label{sec:algo}

%{ \color{purple} Jesus moechte, dass das alles vorher erklaert wird (klar, dazu waren wir bisher einfach noch nicht gekommen). Besonderen Wert hat er %darauf gelegt, dass die 4 Operationen auf den edges so hervorgehoben werden.}

We now present an algorithm that constructs a walk from an initial tree $O$ to a final tree $F$ on the $1$-skeleton of a non-degenerate transportation polytope. 
The walk is fully specified by the corresponding concrete finite sequence of trees, starting with $O$, ending with $F$, where each tree differs 
from the previous one by exactly one edge. We will prove the sequence has at most $N_1+N_2-1-\crit$ steps, which proves the validity of the Hirsch 
conjecture for all transportation polytopes. Before we give a pseudo-code description of the algorithm, let us explain its most important features 
in some detail. The algorithm is based on the following important principle: 

%We find a sequence of 

\begin{center}\em
Construct a sequence of trees, starting from a tree $O$ and ending at a tree $F$, by successively inserting edges contained in $F$ and such that no previously inserted edge is ever deleted.
\end{center}

Note  that in Example \ref{THEex1}, we do not delete an edge that was inserted in a previous step, but it is fine to delete edges of $F$ otherwise. 
To keep track of the edges that may not be deleted, we \emph{shade} them. %Therefore, for each tree $C$ in that sequence, a subset of the edges of $C$ is shaded. We say $C$ is a \emph{partially shaded} tree. \textbf{OK??}
There will be two general situations in which we shade an edge:
\begin{itemize}
\item Whenever we \emph{insert} an edge from $F$ into the current tree, we \emph{shade} it. %We \emph{shade an edge when it is inserted} into the current tree.
\item We may \emph{shade an edge from $F$ that already exists} in the current tree.
\end{itemize}
Note that only the first case corresponds to a step on the skeleton of the transportation polytope from a tree $C$ to a tree $\suc$. 
When inserting an edge into a tree, we also \emph{delete an unshaded edge}. In contrast, we keep the current tree in the second case. %, i.e. we stay at the current vertex on the skeleton 
We will often refer to both of the above situations at the same time. We then write \emph{``(insert and) shade an edge''}. 

In our discussion, we will frequently refer to a tree and the set of shaded edges in the tree at the same time. To this end, we introduce the following two terms:

\begin{itemize}
	\item We call a tree with a (not necessarily strict) subset of edges shaded a \emph{partially shaded tree}.
	\item We call a tree with all edges shaded a \emph{fully shaded tree}. As we only shade edges from $F$ in our algorithm, a fully shaded tree is in fact the tree $F$. % $F$ is the only tree that may be fully shaded.
\end{itemize}

The most important aspect of our algorithm is the order in which edges are (inserted and) shaded; recall again Example \ref{THEex1}. Based on the edge that we insert in the current tree, the margins of the transportation polytope uniquely define which edge is deleted. The order of insertion is determined by the following labeling of the edges in $F$, in which every edge is \emph{labeled $+$ or $-$}. 

\begin{itemize}
	 \item Choose an arbitrary demand node $\delta^*$ and, in $F$, consider all paths starting at $\delta^*$.
	\item Label the edges on these paths alternatingly $+$ and $-$, beginning with a $+$.
%	 \item Odd edges on these paths are labeled $+$.
%	\item Even edges on these paths are labeled $-$. 
\end{itemize}

\begin{figure}[H]
        \centering
            \begin{tikzpicture}[scale=0.5]
\coordinate(v1) at (0,1.5) ;
\coordinate(v2) at (2,-1);
\draw [fill, black] (v1) circle (0.1cm);
\node[above] at (v1) {$\di^*$};
\coordinate (v3) at (0,-1) ;
\coordinate (v4) at (-2,-1) ;
\coordinate (v5) at (3,1.5) ;
\coordinate (v6) at (0,1.5) ;
\draw (v1) -- (v2)  node[midway] {$\;\;\; +$};
\draw  (v1) -- (v3)node[midway] {$\;\;\;  +$};
\draw (v1) -- (v4)node[midway] {$+ \;\;\;$};
\draw  (v2) -- (v5) node[midway] {$\;\; -$};
\coordinate (v8) at (-2,1.5) ;
\coordinate (v7) at (-1,1.5) ;
\coordinate (v9) at (-2,1.5) ;
\draw  (v4) -- (v9)node[midway] {$- \;\;$};
\coordinate (v10) at (-5,1.25) ;
\draw  (v4) -- (v10)node[midway] {$-\;\;\; $};
\coordinate (v11) at (-4,-1) ;

\coordinate (v13) at (-6,-1) ;
\draw  (v10) -- (v11)node[midway] {$ +\;\;\;$};
\draw  (v13) -- (v10)node[midway] {$+\;\;\; $};
\coordinate (v15) at (-5,2) ;
\coordinate (v16) at (-4,2) ;

\coordinate (v17) at (-6,3) ;
\coordinate (v20) at (-10,-1) ;
\coordinate (v18) at (-7,1.25) ;
\coordinate (v19) at (-8,-1) ;
\draw (v17) -- (v4) node[midway] {$\;\;\;\; -$};
\draw  (v18) -- (v13) node[midway] {$-\;\;\;$};
\draw (v18) -- (v19)node[midway] {$+\;\;\;$};
\draw (v17) -- (v20)node[midway] {$+\;\;\;\;$};
\coordinate (v12) at (-6,1.5) ;
\coordinate (v24) at (-7,1.5) ;
\coordinate (v21) at (-8,1.5) ;
\coordinate (v22) at (-10,1.5) ;
\coordinate (v23) at (-11,1.5) ;
\coordinate (v14) at (-9,1.5) ;
\draw  (v20) -- (v22)node[midway] {$-\;\;$};
\draw  (v20) -- (v23) node[midway] {$-\;\;\;$};
%\draw  (v19) -- (v24)node[midway] {$\;\;-$};

\node at (5.5,1.5) {demand nodes};
\node at (5.5,-1) {supply nodes};
			\end{tikzpicture} \caption{A labeling of the edges in $F$.} \label{fig: labels}
\end{figure}

\noindent Figure \ref{fig: labels} is an example for such a labeling. Note that in particular the following properties hold: (a) each supply node is incident to exactly one $+$edge, 
(b) each demand node $\neq \delta^*$ is incident to exactly one $-$edge, and (c) $\delta^*$ is only incident to $+$edges.

These labels for edges in $F$ will not change during our algorithm and can be preprocessed. The main part of the algorithm then begins with an original tree $O$ and all edges unshaded. In each iteration, we (insert and) shade an edge from the final tree $F$. When inserting an edge, the margins tell us which edge will be deleted. We proceed like this until we reach $F$ with all edges shaded.
 
We now take a closer look at a shading step (an iteration of the algorithm). Let $C$ be a partially shaded tree.
%In each iteration, we consider the partially shaded current tree $C$.
% with a (not necessarily strict) subset of the edges already shaded (we say $C$ is \emph{``partially shaded''}). \textbf{OK here???} 
We take a supply node $\si$, one that satisfies a special property in $C$ to be discussed in Section \ref{sec:correct}. We consider the edges that are incident to $\si$ in $F$ and that are still unshaded. We either shade one of these edges from  $C\cap F$, or we insert an edge from $F\backslash C$ and then shade it. 
Thus in each iteration we shade an edge and obtain a succeeding tree $\suc$ with one additional edge shaded. The decision of which edge to (insert and) 
shade is based on the above labeling: For each supply node, its unique incident $+$edge will be the last edge to be shaded. 

Algorithm \ref{algo:hirschwalk} gives a description in pseudo-code of the method: Steps 1 and 2 describe the preprocessing of labeling the edges in $F$ and the initialization of the algorithm. The main loop is stated in Step 3 and we refer to each run through the loop as an {\em iteration} of the algorithm. Note that the actual walk on the skeleton greatly depends on the labeling we fix in the beginning. The chosen labeling has an impact on the if-else clauses in the main loop and also affects the subroutine (Algorithm \ref{algo:newsupply}). Choosing a different labeling, i.e., selecting a
different demand node $\delta^*$ (Step 1 of Algorithm  \ref{algo:hirschwalk}) may change the walk significantly; this might even change the number of steps (the length of the walk on the skeleton). The same holds for choosing a different initial supply node $\si$ (Step 2 of Algorithm  \ref{algo:hirschwalk}).  Thus the constructed walk is not necessarily of minimum length.

\begin{algorithm}[H]
%  \begin{itemize}
%  \item {\bf Input}: Trees $O$ and $F$ corresponding to vertices of a transportation polytope TP
%  \item {\bf Output}:  A sequence $S$ of trees corresponding to a walk from $O$ to $F$ on the skeleton of TP
%  \end{itemize}
\vspace{0.2cm}

 {\bf Input}: Trees $O$ and $F$ corresponding to vertices of an $N_1 {\times} N_2$-transportation polytope $\TP$
\vspace{0.1cm}

\noindent {\bf Output}:  A finite sequence $S$ of trees corresponding to a walk from $O$ to $F$ on the skeleton of  $\TP$. The number of steps described by $S$ is at most $N_1+N_2-1-\mu$, 
where $\mu$ is the number of critical pairs of  $\TP$

  \begin{enumerate}
 \item Choose an arbitrary demand node $\delta^*$ and consider all paths in $F$ starting at $\delta^*$. 	Label the edges on these paths alternatingly $+$ and $-$, beginning with a $+$.
\item Choose an arbitrary supply node $\si$. All edges in $O$ are unshaded. Start with sequence $S$ only containing $O$. Set $C=O$, the current tree of the walk. %$C$ refers to the current tree throughout the iterations.
\item {\bf repeat} 
\begin{itemize}
\item {\bf If} {\em there is an unshaded $-$edge incident to $\si$ in $F$} {\bf then}\\
\hspace*{0.66cm} (insert and) shade an unshaded $-$edge $\sealg=\{\sigma,\delta\}$ incident to $\si$ in $F$ into $C$ to obtain \\
\hspace*{0.66cm} the succeeding tree $\suc$\\
\hspace*{0.66cm} {\bf if} {\em the edge $\sealg$ was inserted (not only shaded)} {\bf then}\\
\hspace*{0.66cm} \hspace*{0.66cm} append  $\suc$ to the end of $S$,\\
\hspace*{0.66cm} \hspace*{0.66cm} set $\di'$ as the demand node incident to the deleted edge\\
\hspace*{0.66cm} {\bf else}\\
\hspace*{0.66cm} \hspace*{0.66cm} set $\di':=\di$ 
\item {\bf Else}\\
\hspace*{0.66cm}  (insert and) shade the unique $+$edge $\sealg=\{\sigma,\delta\}$ incident to $\si$ in $F$ into $C$ to obtain\\
\hspace*{0.66cm} the succeeding tree $\suc$	\\
\hspace*{0.66cm}  {\bf if} {\em the edge $\sealg$ was inserted  (not only shaded)} {\bf then}\\
\hspace*{0.66cm} \hspace*{0.66cm} append  $\suc$ to the end of $S$\\
\hspace*{0.66cm} \hspace*{0.66cm} set $\di'$ as the demand node incident to the deleted edge\\
\hspace*{0.66cm} {\bf else}\\
\hspace*{0.66cm} \hspace*{0.66cm} set $\di':=\di$.\\
\item {\bf If} {\em $\delta^*$ is only incident to shaded edges in $\suc$} {\bf then} \\
\hspace*{0.66cm} \hspace*{0.66cm} {\bf return }$S$ and {\bf stop}
%\item Update $\si$ by calling {\bf Find New Supply Node $\si$} (Algorithm \ref{algo:newsupply}) with input $\di'$ and $\suc$
\item Update $\si$ by calling \emph{Algorithm \ref{algo:newsupply}} with input $\di'$ and $\suc$
\item Set $C:=\suc$
\end{itemize} 
{\bf end repeat}
\end{enumerate}
  \caption{Hirsch-walk in  a Transportation Polytope}
   \label{algo:hirschwalk}
\end{algorithm}

\begin{algorithm}[H]
% \begin{itemize}
%  \item {\bf Input}: Demand node $\di'$ (and final tree $F$ with labeled edges)
%  \item {\bf Output}: Supply node $\si$ 
%  \end{itemize}
\vspace{0.2cm}

 {\bf Input}: Demand node $\di'$ and a partially shaded tree $\suc$ (with $+$/$-$labels for the shaded edges), corresponding to a vertex of  $\TP$ 
 
 % (and final tree $F$ with labeled edges)
\vspace{0.1cm}

\noindent {\bf Output}:  A new supply node $\si$ in $\suc$

\begin{itemize}
\item {\bf if} {\em there is an unshaded edge $\{\si',\di'\}$ in $\suc$}  {\bf then} \\
\hspace*{0.66cm}{\bf return }$\si=\si'$ 
\item {\bf else}\\
\hspace*{0.66cm}set $\si''$ as the supply node of the unique $-$edge  $\{\si'',\di'\}$ incident to $\di'$\\
\hspace*{0.66cm}{\bf return }$\si=\si''$ 
\end{itemize} 
  \caption{Find New Supply Node $\si$.}
  \label{algo:newsupply}
\end{algorithm}

\noindent Let us illustrate how the algorithm works by applying it to  Example \ref{THEex1}. 
\begin{example}[Example \ref{THEex1} revisited] \label{ex: steps of alg}
We show a run of Algorithm \ref{algo:hirschwalk} for the instance from Example \ref{THEex1}. As before, we want to connect $O$ and $F$ as depicted in Figure \ref{fig: O and F} with a walk on the skeleton of the transportation polytope. Figure \ref{fig: O and F} also shows a choice of $\delta^*$ and the corresponding labeling of the edges in $F$ according to Step 1 in Algorithm \ref{algo:hirschwalk}.

	\begin{figure}%[H]
		\centering
			\begin{tikzpicture}[scale=0.55]
					\coordinate (c1) at (1,0);
					\coordinate (c2) at (3,0);
					\coordinate (x1) at	(0,2);
					\coordinate (x2) at (2,2);
					\coordinate (x3) at (4,2);
%					\draw [fill, black] (c1) circle (0.1cm);
%					\draw [fill, black] (c2) circle (0.1cm);
%					\draw [fill, black] (x1) circle (0.1cm);
%					\draw [fill, black] (x2) circle (0.1cm);
%					\draw [fill, black] (x3) circle (0.1cm);
%					\node[below] at (c1) {$3$};
%					\node[below] at (c2) {$3$};
%					\node[above] at (x1) {$2$};
%					\node[above] at (x2) {$2$};
%				  \node[above] at (x3) {$2$};
				  \draw (c1)--(x2) ;
				  \draw (c2)--(x2) ;
					\draw (c1)--(x1);
				  \draw (c2)--(x3) ;			  
					\node at (2,-1.5) {tree $O$};
			\end{tikzpicture}
\hspace*{2cm}
			\begin{tikzpicture}[scale=0.55]
					\coordinate (c1) at (1,0);
					\coordinate (c2) at (3,0);
					\coordinate (x1) at	(0,2);
					\coordinate (x2) at (2,2);
					\coordinate (x3) at (4,2);
%
%					\node[below] at (c1) {$3$};					\draw [fill, black] (c1) circle (0.1cm);
%					\draw [fill, black] (c2) circle (0.1cm);
%					\draw [fill, black] (x1) circle (0.1cm);
%					\draw [fill, black] (x2) circle (0.1cm);
%					\draw [fill, black] (x3) circle (0.1cm);
%					\node[below] at (c2) {$3$};
					\node[above] at (x1) {$\delta^*$};
%					\node[above] at (x2) {$2$};
%				  \node[above] at (x3) {$2$};
				  \draw (c1)--(x3) node[near end] {$-$};
				  \draw (c1)--(x2)node[near end] {$+\; $};
				  \draw (c2)--(x1)node[near end] {$+$};
					\draw (c2)--(x2)node[near end] {$\; -$};
					\node at (2,-1.5) {tree $F$};
			\end{tikzpicture}
		\caption{Initial tree $O$ with all edges unshaded and final tree $F$ with edge labels.} \label{fig: O and F}

\end{figure}

	\begin{figure}%[H]
		\centering
			\begin{tikzpicture}[scale=0.55]
					\coordinate (c1) at (1,0);
					\coordinate (c2) at (3,0);
					\coordinate (x1) at	(0,2);
					\coordinate (x2) at (2,2);
					\coordinate (x3) at (4,2);
%					\draw [fill, black] (c1) circle (0.1cm);
%					\draw [fill, black] (c2) circle (0.1cm);
%					\draw [fill, black] (x1) circle (0.1cm);
%					\draw [fill, black] (x2) circle (0.1cm);
%					\draw [fill, black] (x3) circle (0.1cm);
					\node[below] at (c1) {$\si$\color{white}{$'$}};
					\node[below] at (c2) {$\si'$};
					\node[above] at (x2) {$\di'$};
				  \node[above] at (x3) {$\di$};
				  \draw (c1)--(x2) node[midway] {$\backslash$} ;
				  \draw (c2)--(x2) ;
					\draw (c1)--(x1);
				  \draw (c2)--(x3) ;
				  \draw[dashed, line width = 1.5] (c1)--(x3);			  
					\node at (2,-1.5) {iteration 1};
					\node at (2,-2.5) {pivot 1};
			\end{tikzpicture}
\hspace*{1cm}
			\begin{tikzpicture}[scale=0.55]
					\coordinate (c1) at (1,0);
					\coordinate (c2) at (3,0);
					\coordinate (x1) at	(0,2);
					\coordinate (x2) at (2,2);
					\coordinate (x3) at (4,2);
%					\draw [fill, black] (c1) circle (0.1cm);
%					\draw [fill, black] (c2) circle (0.1cm);
%					\draw [fill, black] (x1) circle (0.1cm);
%					\draw [fill, black] (x2) circle (0.1cm);
%					\draw [fill, black] (x3) circle (0.1cm);
					\node[below] at (c2) {$\si=\si''$};
%					\node[below] at (c2) {$3$};
%					\node[above] at (x1) {$2$};
%					\node[above] at (x2) {$\di'$};
				  \node[above] at (x2) {$\di=\di'$};
				  \draw (c2)--(x2) ;
					\draw (c1)--(x1);
				  \draw (c2)--(x3) ;
				  \draw[line width = 1.5] (c1)--(x3);			  
				  \draw[dashed, line width = 1.5] (c2)--(x2);			  
					\node at (2,-1.5) {iteration 2};
					\node at (2,-2.8) {$ $};
			\end{tikzpicture}
\hspace*{1cm}
			\begin{tikzpicture}[scale=0.55]
					\coordinate (c1) at (1,0);
					\coordinate (c2) at (3,0);
					\coordinate (x1) at	(0,2);
					\coordinate (x2) at (2,2);
					\coordinate (x3) at (4,2);
%					\draw [fill, black] (c1) circle (0.1cm);
%					\draw [fill, black] (c2) circle (0.1cm);
%					\draw [fill, black] (x1) circle (0.1cm);
%					\draw [fill, black] (x2) circle (0.1cm);
%					\draw [fill, black] (x3) circle (0.1cm);
					\node[below] at (c2) {$\si$\color{white}{$'$}};
					\node[below] at (c1) {$\si''$};
%					\node[above] at (x1) {$2$};
					\node[above] at (x3) {$\di'$};
				  \node[above] at (x1) {$\di$};
				  \draw[line width = 1.5] (c2)--(x2) ;
					\draw (c1)--(x1);
				  \draw (c2)--(x3) node[midway] {$\backslash$} ; ;
				  \draw[line width = 1.5] (c1)--(x3);			  
				  \draw[dashed, line width = 1.5] (c2)--(x1);			  
					\node at (2,-1.5) {iteration 3};
					\node at (2,-2.5) {pivot 2};
			\end{tikzpicture}
\hspace*{1cm}
			\begin{tikzpicture}[scale=0.55]
					\coordinate (c1) at (1,0);
					\coordinate (c2) at (3,0);
					\coordinate (x1) at	(0,2);
					\coordinate (x2) at (2,2);
					\coordinate (x3) at (4,2);
%					\draw [fill, black] (c1) circle (0.1cm);
%					\draw [fill, black] (c2) circle (0.1cm);
%					\draw [fill, black] (x1) circle (0.1cm);
%					\draw [fill, black] (x2) circle (0.1cm);
%					\draw [fill, black] (x3) circle (0.1cm);
					\node[below] at (c1) {$\si$\color{white}{$'$}};
%					\node[below] at (c2) {$3$};
%					\node[above] at (x1) {$2$};
					\node[above] at (x1) {$\di'=\di^*$};
				  \node[above] at (x2) {$\di$};
				  \draw[line width = 1.5] (c2)--(x2) ;
					\draw (c1)--(x1) node[midway] {$\slash$} ;;
				  \draw[line width = 1.5] (c1)--(x3);
				  \draw[ line width = 1.5] (c2)--(x1);
				  \draw[dashed, line width = 1.5] (c1)--(x2);		  
					\node at (2,-1.5) {iteration 4};
					\node at (2,-2.5) {pivot 3};
			\end{tikzpicture}
		\caption{The iterations of Algorithm \ref{algo:hirschwalk} for the input from Figure \ref{fig: O and F}.} \label{fig: ex algo}

\end{figure}

The algorithm inserts and shades edges in four iterations of Algorithm \ref{algo:hirschwalk}, which are depicted in Figure \ref{fig: ex algo}.  As we will see, iterations 1,3, and 4 correspond to actual steps on the skeleton of the underlying transportation polytope. In contrast, iteration 2 illustrates a situation where we remain at the current tree. The names of supply nodes and demand nodes in the figure correspond to the notation used in Algorithms \ref{algo:hirschwalk} and \ref{algo:newsupply}. In particular, recall that in each iteration we (insert and) shade an edge incident to the carefully chosen current supply node $\si$. Further, recall that $\di$ refers to the demand node of this edge. Let us follow a run of Algorithm \ref{algo:hirschwalk} for this example in some detail. The reader should keep in mind that in all calculations we use the margins of Example \ref{THEex1} where supply  nodes (bottom row) have supply $3$ and demand nodes (top row) have demand $2$.

In the first iteration, the node $\si$ is an arbitrary supply node chosen in Step 2 of the algorithm. We begin with all edges unshaded. By inserting and shading the $-$edge $\{\si,\di\}$, we delete the edge $\{\si,\delta'\}$, which is an edge incident to $\di'$. It is important to note that the margins  dictate which edge is deleted after one is inserted. The demand node $\di'$ is the input for the first call of the subroutine to find a new supply node (Algorithm \ref{algo:newsupply}). There still is the unshaded edge $\{\di',\si'\}$ incident to $\di'$ in the succeeding tree, so the subroutine returns the supply node $\si'$ and we continue with it as the new supply node $\si$. %Note that we inserted an edge during this first iteration, so we performed a step on the skeleton.

%In the second iteration, there is a $-$edge $\{\si,\di\}$ incident to the new $\si$ in $F$ that already exists in the current tree but is still unshaded, so in this step we can choose this edge $\{\si,\di\}$ and only shade it (no insertion). 
%\vskip 1cm
%I would fix with . and ,
%\vskip 1cm
%In the second iteration, there is a $-$edge $\{\si,\di\}$ incident to the new $\si$ in $F$ that already exists in the current tree, but is still unshaded. In the current step we choose $\{\si,\di\}$ and only shade it (no insertion). 
%}

In the second iteration, there is a $-$edge $\{\si,\di\}$ incident to the new $\si$ in $F$ that already exists in the current tree, but is still unshaded. We choose $e=\{\si,\di\}$ and only shade it (no insertion). 
Therefore we have $\di'=\di$. Since there is no unshaded edge incident to $\di'$, the update subroutine returns the supply node $\si''$ of the unique $-$edge incident to $\di'$ as the new $\si$ for the next iteration. (In this case, it is the same supply node as for the previous iteration.)

In the third iteration, there is only one edge $\{\si,\di\}$ (a $+$edge) left to shade incident to the new $\si$. We insert and shade it. Once more dictated by the margins at nodes, this forces the deletion of the unshaded edge incident to $\di'$. There are no unshaded edges incident to $\di'$ and thus we continue with $\si''$, the supply node of the unique $-$edge incident to $\delta'$. We update $\si:=\si''$.

In the fourth iteration, we again insert the only unshaded edge $\{\si,\di\}$ (a $+$edge) incident to $\si$. Given the margins of the instance, this forces the deletion of the unshaded edge incident to $\di'=\di^*$. Now $\di^*$ is only incident to shaded edges in the new tree $\suc$. This is the stopping criterion for Algorithm \ref{algo:hirschwalk}. We have reached the final tree $F$ with all edges shaded. \hfill\qed
\end{example}

%The remainder of this paper %{\bf \; TODO: replace 'paper' by 'section', if we add a proof for the corollary!\; } 
%is dedicated to proving the correctness of the algorithm.

\section{Proof of Correctness}\label{sec:correct}

We now turn to the proof of correctness for Algorithm \ref{algo:hirschwalk}.  First, we introduce notation and two 
properties for trees, respectively their nodes, that are at the core of the proof.

\subsection{Well-Connectedness and the UNO and SIN properties}

We distinguish two states for each node and shaded edge in the current tree. 

\begin{definition}
Let $C$ be the partially shaded current tree and $F$ the final tree. 

We say a supply node  $\si$ is \emph{well-connected} if all edges that are incident to $\si$ in 
$F$ exist in the current tree $C$ and are shaded ('nothing left to insert'). Otherwise it is \emph{open}. 

A demand node $\di$ is \emph{well-connected} if it is not incident to an unshaded edge in $C$ ('nothing to delete'). 
Otherwise it is \emph{open}.

A shaded edge incident to at least one well-connected node is a \emph{well-connected edge}. 
\end{definition}

Note that to detect well-connectedness during a run of Algorithm \ref{algo:hirschwalk}, it is enough to consider the current tree and the edge labels from the final tree $F$. 
For the demand nodes, this is obvious. A supply node is well-connected if and only if it is incident to a shaded $+$edge. This is because 
the Algorithm  \ref{algo:hirschwalk} first (inserts and) shades all $-$edges incident to a supply node and only then the unique $+$edge. Further, shaded edges are never deleted.

The well-connected edges together with their incident nodes form an important structure for our proofs.

\begin{definition}
A \emph{well-connected component} in a partially shaded tree $C$ is a connected subgraph induced by the well-connected edges. Further, 
any node not incident to a well-connected edge forms a component by itself.
\end{definition}

The well-connected components are exactly the connected components of the forest we would obtain deleting all edges that are not well-connected 
from the current tree. Note that every node and every well-connected edge belongs to a unique well-connected component. Also, every shaded $+$edge 
belongs to some well-connected component (its supply node is well-connected). However, this is not always the case for shaded $-$edges as its supply 
and demand node might both be open (edges to insert / unshaded edges); see  Example \ref{ex: components} for a illustration of these concepts.

In what follows we will simply use ``component(s)'' when we talk about the well-connected components. 

\begin{definition}
We say a component of a partially shaded tree $C$ satisfies \emph{property  (UNO)} 
(\emph{un}ique \emph{o}pen node property)  if it contains a unique  open node. 
(UNO) holds in a partially shaded tree $C$ if all its components satisfy (UNO). 
\end{definition}

We will show that Algorithm \ref{algo:hirschwalk} always preserves property (UNO) until we reach the final tree $F$ with 
all edges shaded and all nodes well-connected (see Lemma \ref{lem: done}).  This property is crucial for our arguments that follow.

%%%%%%%%%%%%%%%%%%%%
%
%
%{\bf WORK UP TO HERE!!!! REVISE!!!! MAKE SURE IT READS PERFECT SO FAR!!!! CHECK THE NEW SECTION TOO}
%
%\newpage
%%%%%%%%%%%%%%%%%%%%%%%%%%

\begin{example}[UNO]\label{ex: components}
In this example, we illustrate several well-connected components and their open/well-connected nodes.  The bold edges are shaded, ovals indicate the components. Recall that all shaded edges are contained in the final tree $F$ and thus have a $+$ or $-$ label. A supply node is well-connected if and only if it is incident to a shaded $+$edge, while a demand node is well-connected if and only if there are no unshaded edges incident to this node.

In the left-hand component in Figure \ref{fig: ex1}, the supply node $\si$ is open as there is no shaded $+$edge incident to $\si$ and therefore there is an edge left to insert. All  other nodes are well-connected. In the right-hand component, the demand node $\di$ is the unique open node as it is incident to an unshaded edge. In particular, both components satisfy (UNO). 
\begin{figure}[H]
\center
\begin{tikzpicture}[scale= 0.6]
\coordinate (v1) at (0,2);
\coordinate (v2) at (1,0);
\coordinate (v3) at (2,2);
\coordinate (v4) at (2,0);
\coordinate (v5) at (3,0);
\coordinate (v6) at (4,2);
\coordinate (v7) at (4,0);

\draw (v2)--(0.7,-0.7);
\draw (v4)--(1.7,-0.7);
\draw (v4)--(2.3,-0.7);
\draw (v5)--(3.3,-0.7);
\draw (v7)--(4.3,-0.7);

\draw[line width = 1.5] (v1)--(v2) node[midway] {$- \;\;$};
\draw[line width = 1.5] (v2)--(v3) node[midway] {$- \;\; $};
\draw[line width = 1.5] (v3)--(v4) node[midway] {$ \;\;+$};
\draw[line width = 1.5] (v3)--(v5) node[midway] {$ \;\;\ +$};
\draw[line width = 1.5] (v5)--(v6) node[midway] {$ -\;\;$};
\draw[line width = 1.5] (v6)--(v7) node[midway] {$ \;\;\; +$};
\node[below] at (v2) {$\; \si$};
%\draw  (v2) circle (0.1cm);
\draw[line width = 0.05]  (2.1, 1.1) ellipse (2.9 and 1.7);
\end{tikzpicture}
%\caption{A component with open node $\si$ (no $+$edge and thus an edge left to insert).}
%\end{figure}
%
%
%\begin{figure}[H]
%\center
\hspace*{2cm}
\begin{tikzpicture}[scale= 0.6]
\coordinate (v1) at (0,2);
\coordinate (v2) at (1,0);
\coordinate (v3) at (2,2);
\coordinate (v4) at (2,0);
\coordinate (v5) at (3,0);
\coordinate (v6) at (3.8,2);
\coordinate (v7) at (4.8,2);

\draw (0.2,0)--(-0.1,-0.7);
\draw (v2)--(0.7,-0.7);
\draw (v4)--(1.7,-0.7);
\draw (v4)--(2.3,-0.7);
\draw (v5)--(3.3,-0.7);

\draw[line width = 1.5] (v1)--(0.2,0) node[midway] {$+ \;\;\;$};
\draw[line width = 1.5] (v1)--(v2) node[midway] {$ \;\;\;+$};
\draw[line width = 1.5] (v2)--(v3) node[midway] {$- \;\; $};
\draw[line width = 1.5] (v3)--(v4) node[midway] {$ \;\;+$};
\draw[line width = 1.5] (v3)--(v5) node[midway] {$ \;\;\ +$};
\draw[line width = 1.5] (v5)--(v6) node[midway] {$ -\;\;$};
\draw[line width = 1.5] (v5)--(v7) node[midway] {$ \;\;\;\ -$};
\draw (v1)--(-1.5,0);

\node[above] at (v1) {$\di$};
%\draw  (v1) circle (0.1cm);
\draw[line width = 0.05]  (2.2, 1.35) ellipse (3.2 and 1.9);
\end{tikzpicture}\caption{Components with unique open nodes $\si$ and $\di$, respectively.} \label{fig: ex1}
\end{figure}
Figure \ref{fig: ex2} illustrates a configuration in which two components are connected by a shaded $-$edge $\{\si,\di\}$. 
This edge is not well-connected because both endpoints are open: $\si$ is not incident to a $+$edge while $\di$ is incident to 
an unshaded edge. Note that (UNO) holds in both components.
\begin{figure}[H]
\center
\begin{tikzpicture}[scale= 0.6]
\coordinate (v1) at (-0.5,2);
\coordinate (v2) at (0,0);
\coordinate (v3) at (2.2,2);
\coordinate (v4) at (1.3,0);
\coordinate (v5) at (3,0);
\coordinate (v6) at (4,2);
\coordinate (v7) at (4,0);

\draw (v2)--(-0.3,-0.7);
\draw (v5)--(3.3,-0.7);
\draw (v7)--(4.3,-0.7);

\draw[line width = 1.5] (v1)--(v2) node[midway] {$- \;\;$};
\draw[line width = 1.5] (v2)--(v3) node[midway] {$- \;\; $};
\draw (v3)--(v4);
\draw[line width = 1.5] (v3)--(v5) node[midway] {$ \;\; +$};
\draw[line width = 1.5] (v5)--(v6) node[midway] {$ -\;\;$};
\draw[line width = 1.5] (v6)--(v7) node[midway] {$ \;\;\; +$};
\node[below] at (v2) {$\; \si$};
%\draw  (v2) circle (0.1cm);
\node[above] at (v3) {$\di$};
%\draw  (v3) circle (0.1cm);

\draw[line width = 0.05]  (-0.2, 1) ellipse (0.8 and 1.5);
\draw[line width = 0.05]  (3.2, 1.1) ellipse (1.7 and 1.5);
\end{tikzpicture}\caption{Two components satisfying (UNO), connected by a shaded edge.} \label{fig: ex2}
\end{figure}

Finally, Figure \ref{fig: no (UNO)} depicts two components that do not satisfy (UNO): They have two open nodes each, 
$\si$ and $\di$, respectively $\di$ and $\di'$. 
Observe that here all shaded edges are well-connected as they are incident to a well-connected supply or demand node. 

\begin{figure}[H]
\center
\begin{tikzpicture}[scale= 0.6]
\coordinate (v1) at (0,2);
\coordinate (v2) at (1,0);
\coordinate (v3) at (2,2);
\coordinate (v4) at (4,-1);
\coordinate (v5) at (3,0);
\coordinate (v6) at (4,2);

\draw[line width = 1.5] (v1)--(v2) node[midway] {$- \;\;\; $};
\draw[line width = 1.5] (v2)--(v3) node[midway] {$- \;\; $};
\draw (v6)--(v4);
\draw[line width = 1.5] (v3)--(v5) node[midway] {$ \;\;\ +$};
\draw[line width = 1.5] (v5)--(v6) node[midway] {$\; \;\; -$};
\node[below] at (v2) {$\si$};
%\draw (v2) circle (0.1cm);
\node[above] at (v6) {$\di$};
%\draw  (v6) circle (0.1cm);

\draw[line width = 0.05]  (2, 1.1) ellipse (3 and 1.5);
\end{tikzpicture}
\hspace*{2cm}
\begin{tikzpicture}[scale= 0.6]
\coordinate (v1) at (0,2);
\coordinate (v2) at (1,0);
\coordinate (v3) at (2,2);
\coordinate (v4) at (2,-1);
\coordinate (v5) at (3,0);
\coordinate (v6) at (4,2);

\draw[line width = 1.5] (v1)--(v2) node[midway] {$+ \;\;\; $};
\draw[line width = 1.5] (v2)--(v3) node[midway] {$- \;\; $};
\draw (v3)--(v4);
\draw (v1)--(-0.5,-1);
\draw[line width = 1.5] (v3)--(v5) node[midway] {$ \;\;\ +$};
\draw[line width = 1.5] (v5)--(v6) node[midway] {$\; \;\; -$};
\node[above] at (v1) {$\di$};
%\draw  (v1) circle (0.1cm);
\node[above] at (v3) {$\di'$};
%\draw  (v3) circle (0.1cm);

\draw[line width = 0.05]  (2, 1.1) ellipse (3 and 1.5);
\end{tikzpicture}
\caption{Components with two open nodes each ($\si$ and $\di$ / $\di$ and $\di'$).} \label{fig: no (UNO)}
\end{figure}
\vspace*{-0.75cm}\hfill\qed
\end{example}

Besides (UNO), we need a second property. It is concerned with the supply node $\si$ that is incident to the edge we (insert and) shade. We distinguish odd and even edges 
with respect to the node $\sigma$ in the current tree $C$: Every edge of the current tree $C$ lies on a unique path starting at 
supply node $\sigma$. We number the edges on these paths, where the edges incident to $\sigma$ are the first edges. Then an edge is \emph{odd} with respect to $\si$ if 
it has an odd number, otherwise it is \emph{even}. In particular, these paths alternate between odd and even edges. The property (SIN) now imposes a condition 
on the odd edges on the paths starting at $\si$.

\begin{definition}
A supply node $\si$ satisfies \emph{property (SIN)} (\emph{s}upply node \emph{in}sertion property) in a partially shaded tree $C$ if all edges that are odd with respect 
to $\si$ are unshaded edges or (shaded) $-$edges incident to well-connected demand nodes.
\end{definition}

The odd edges on paths starting at a supply node $\si$ are of particular interest: They are the only edges that could theoretically be deleted when inserting an edge incident to $\si$. Note again that the margins determine which odd edge is deleted. However, as we will see, by following Algorithms \ref{algo:hirschwalk} and \ref{algo:newsupply} the odd edges on paths starting at $\si$ that are already shaded will not be deleted. In fact, we will prove later that Algorithm \ref{algo:newsupply} returns a supply node with the (SIN) property. Proving these two statements is a key part of our upcoming proof.

Typically, we will only distinguish between odd and even edges on paths starting at a fixed supply node $\si$ with the (SIN) property. %when showing that $\si$ satisfies the (SIN) property. 
To have a simpler wording, we will refer to these edges only as \emph{odd} or \emph{even} if it is clear which supply  node is being considered.

\begin{example}[SIN]\label{ex: SIN} 
We illustrate the property (SIN) with the tree depicted in Figure \ref{fig: exsin}. Edges highlighted by wavy lines are the odd edges on paths starting at $\sigma$. 
These edges are either unshaded or they are shaded (bold) $-$edges with well-connected demand node ($\delta^1$ and $\delta^2$  in Figure \ref{fig: exsin} 
are not incident to an unshaded edge and thus well-connected). There are no conditions on the even edges on paths starting at $\sigma$.

\begin{figure}
\center
\begin{tikzpicture}[scale= 0.8]
\coordinate (v1) at (0,2);
\coordinate (v2) at (1,0);
\coordinate (v3) at (2,2);
\coordinate (v4) at (2,0);
\coordinate (v5) at (3,0);
\coordinate (v6) at (4.5,2);
\coordinate (v7) at (5,0);
\coordinate (v8) at (3,2);
\coordinate (v9) at (-1,0);
\coordinate (v10) at (-2,1.5);
\coordinate (v11) at (-2,0);
\coordinate (v12) at (-3,0);
\coordinate (v13) at (-1.8,2.3);
\coordinate (v14) at (-4,0);

\draw[gray, snake=coil, segment aspect=0] (v1)--(v2);
\draw[gray, snake=coil, segment aspect=0] (v2)--(v3);
\draw[gray, snake=coil, segment aspect=0] (v5)--(v6);
\draw[gray, snake=coil, segment aspect=0] (v5)--(v8);
\draw[gray, snake=coil, segment aspect=0] (v9)--(v10);
\draw[gray, snake=coil, segment aspect=0] (v9)--(v13);

\draw[line width = 1.5] (v1)--(v2) node[midway] {$ -\;\;$} ;
\draw (v5)--(v8);
\draw (v2)--(v3);
\draw (v3)--(v4);
\draw (v3)--(v5);
\draw[line width = 1.5] (v5)--(v6) node[midway] {$\;\; - $};
\draw[line width = 1.5 ] (v6)--(v7)node[midway] {$\;\;+$};
\draw[line width = 1.5] (v10)--(v11) node[midway] {$ -\;\;$};

\draw[line width = 1.5]  (v1)--(v9)  node[midway] {$ +\;\;$};
\draw (v9)--(v10);
\draw (v10)--(v11);
\draw (v10)--(v12);
\draw (v9)--(v13);
\draw (v13)--(v14);

\node[below] at (v2) {$ \si$};
\node[above] at (v1) {$\di^1$};
%\node[above] at (v3) {$\di^2$};
\node[above] at (v6) {$\di^2$};
\end{tikzpicture}
\caption{A node $\sigma$  with the (SIN) property.} \label{fig: exsin}
\end{figure}

Recall that the odd (wavy) edges are the ones that might be deleted when inserting an edge incident to $\sigma$. However, (UNO), (SIN), 
and our insertion strategy (first $-$edges, then $+$edge) imply that we always delete an unshaded edge. In particular, a shaded $-$edge will 
not be deleted if it is incident to a well-connected demand node.
\hfill\qed
\end{example}

\begin{example}[Example \ref{ex: steps of alg} revisited] \label{ex: compchanges}
We illustrate the well-connected nodes and components throughout the run of the algorithm in Example \ref{ex: steps of alg} in Figure \ref{ex: compchanges}.

	\begin{figure}
		\centering
			\begin{tikzpicture}[scale=0.6]
%	\useasboundingbox (0,-2) rectangle (4,2) ;
					\coordinate (c1) at (1,0);
					\coordinate (c2) at (3,0);
					\coordinate (x1) at	(0,2);
					\coordinate (x2) at (2,2);
					\coordinate (x3) at (4,2);
					\draw[line width = 0.05]  (c1) ellipse (0.5 and 0.5);
					\draw[line width = 0.05]  (c2) ellipse (0.5 and 0.5);
					\draw[line width = 0.05]  (x1) ellipse (0.5 and 0.5);
					\draw[line width = 0.05]  (x2) ellipse (0.5 and 0.5);
					\draw[line width = 0.05]  (x3) ellipse (0.5 and 0.5);
%					\draw [fill, black] (c1) circle (0.1cm);
%					\draw [fill, black] (c2) circle (0.1cm);
%					\draw [fill, black] (x1) circle (0.1cm);
%					\draw [fill, black] (x2) circle (0.1cm);
%					\draw [fill, black] (x3) circle (0.1cm);
%					\node[below] at (c1) {$\si$};
%				  \node[above] at (x3) {$\di$};
				  \draw (c1)--(x2) ;
				  \draw (c2)--(x2) ;
					\draw (c1)--(x1);
				  \draw (c2)--(x3) ;
					\node at (2,-1.5) {start ($O$)};
			\end{tikzpicture}
\hspace*{1cm}
			\begin{tikzpicture}[scale=0.6]
%	\useasboundingbox (0,-2) rectangle (4,2) ;
					\coordinate (c1) at (1,0);
					\coordinate (c2) at (3,0);
					\coordinate (x1) at	(0,2);
					\coordinate (x2) at (2,2);
					\coordinate (x3) at (4,2);
					\draw[line width = 0.05]  (c1) ellipse (0.5 and 0.5);
					\draw[line width = 0.05]  (c2) ellipse (0.5 and 0.5);
					\draw[line width = 0.05]  (x1) ellipse (0.5 and 0.5);
					\draw[line width = 0.05]  (x2) ellipse (0.5 and 0.5);
					\draw[line width = 0.05]  (x3) ellipse (0.5 and 0.5);
%					\draw [fill, black] (c1) circle (0.1cm);
%					\draw [fill, black] (c2) circle (0.1cm);
%					\draw [fill, black] (x1) circle (0.1cm);
%					\draw [fill, black] (x2) circle (0.1cm);
%					\draw [fill, black] (x3) circle (0.1cm);
%					\node[below] at (c2) {$\si$};
%				  \node[above] at (x2) {$\di$};
				  \draw (c2)--(x2) ;
					\draw (c1)--(x1);
				  \draw (c2)--(x3) ;
				  \draw[line width = 1.5] (c1)--(x3) node[near end]{$-\;\;$};			  
					\node at (2,-1.5) {after iteration 1};
			\end{tikzpicture}
\hspace*{1cm}
			\begin{tikzpicture}[scale=0.6]
%	\useasboundingbox (0,-2) rectangle (4,2) ;
					\coordinate (c1) at (1,0);
					\coordinate (c2) at (3,0);
					\coordinate (x1) at	(0,2);
					\coordinate (x2) at (2,2);
					\coordinate (x3) at (4,2);
					\draw[line width = 0.05]  (c1) ellipse (0.5 and 0.5);
					\draw[line width = 0.05]  (x1) ellipse (0.5 and 0.5);
					\draw[line width = 0.05]  (x3) ellipse (0.5 and 0.5);
					\draw[line width = 0.05, rotate= 30]  (2.6,-0.3) ellipse (0.7 and 1.7);
%					\draw [fill, black] (c1) circle (0.1cm);
%					\draw [fill, black] (c2) circle (0.1cm);
%					\draw [fill, black] (x1) circle (0.1cm);
					\draw [fill, black] (x2) circle (0.15cm);
%					\draw [fill, black] (x3) circle (0.1cm);
%					\node[below] at (c2) {$\si$};
%				  \node[above] at (x1) {$\di$};
%				  \draw[line width = 1.5] (c2)--(x2) ;
					\draw (c1)--(x1);
				  \draw (c2)--(x3) ;
				  \draw[line width = 1.5] (c1)--(x3) node[near end]{$-\;\;$};	
				   \draw[line width = 1.5] (c2)--(x2)node[near end]{$\;\;-$};			  
				\node at (2,-1.5) {after iteration 2};
			\end{tikzpicture}
\hspace*{1cm}

			\begin{tikzpicture}[scale=0.6]
					\coordinate (c1) at (1,0);
					\coordinate (c2) at (3,0);
					\coordinate (x1) at	(0,2);
					\coordinate (x2) at (2,2);
					\coordinate (x3) at (4,2);
					\draw[line width = 0.05, rotate= -57]  (0.5,2.3) ellipse (0.5 and 2.7);
					\draw[line width = 0.05, rotate= 55]  (2.2,-0.8) ellipse (1.1 and 2.4);
%					\draw [fill, black] (c1) circle (0.1cm);
					\draw [fill, black] (c2) circle (0.15cm);
%					\draw [fill, black] (x1) circle (0.1cm);
					\draw [fill, black] (x2) circle (0.15cm);
					\draw [fill, black] (x3) circle (0.15cm);
%					\node[below] at (c1) {$\si$};
%%					\node[below] at (c2) {$3$};
%%					\node[above] at (x1) {$2$};
%					\node[above] at (x1) {$\di'=\di^*$};
%				  \node[above] at (x2) {$\di$};
				  \draw[line width = 1.5] (c2)--(x2)node[near end]{$\;\;-$};
					\draw (c1)--(x1);
				  \draw[line width = 1.5] (c1)--(x3) node[near end]{$-\;\;$};	
				  \draw[ line width = 1.5] (c2)--(x1)node[near end]{$\;\;\, +$};	
					\node at (2,-1.5) {after iteration 3};
			\end{tikzpicture}
\hspace*{1cm}
			\begin{tikzpicture}[scale=0.6]
					\coordinate (c1) at (1,0);
					\coordinate (c2) at (3,0);
					\coordinate (x1) at	(0,2);
					\coordinate (x2) at (2,2);
					\coordinate (x3) at (4,2);
					\draw [fill, black] (c1) circle (0.15cm);
					\draw [fill, black] (c2) circle (0.15cm);
					\draw [fill, black] (x1) circle (0.15cm);
					\draw [fill, black] (x2) circle (0.15cm);
					\draw [fill, black] (x3) circle (0.15cm);
%					\node[below] at (c1) {$\si$};
%%					\node[below] at (c2) {$3$};
%%					\node[above] at (x1) {$2$};
%					\node[above] at (x1) {$\di'=\di^*$};
%				  \node[above] at (x2) {$\di$};
				  \draw[line width = 1.5] (c2)--(x2)node[near end]{$\;\;-$};
					\draw[line width = 1.5]  (c1)--(x2)node[near end]{$+\;\;$};	
				  \draw[line width = 1.5] (c1)--(x3) node[near end]{$-\;\;$};	
				  \draw[ line width = 1.5] (c2)--(x1)node[near end]{$\;\;\, +$};	
					\node at (2,-1.5) {after iteration 4 ($F$)};
					\draw[line width = 0.05]  (2,1.1) ellipse (3 and 1.7);
			\end{tikzpicture}
					\caption{Components (encircled) and well-connected nodes (filled) throughout a run of Algorithm \ref{algo:hirschwalk}. } \label{fig: ex algo components}
\end{figure}

Note that not all shadings affect the components (iteration 1). In the other iterations, smaller components are successively merged to larger ones as nodes and edges become well-connected. Well-connected components can be joined

%\newpage \noindent Components can be joined  

%\newpage

\begin{itemize}
\item when a supply node becomes well-connected by (inserting and) shading the $+$edge incident to the node (see bottom right node in iteration 3) or
\item when the last unshaded edge incident to a demand node is shaded (top center in iteration 2) or deleted (top right in iteration 3).
\end{itemize} Then the open node of one of the components becomes the open node of the new, larger component. In particular, the trees satisfy (UNO) up to after iteration $3$.
\hfill\qed

\end{example}

We close this subsection by proving two lemmas, Lemma \ref{lem: done} and \ref{lem: not all +}, which will be useful later in the main steps of the proof.

Recall that our algorithm has to stop if $C$ is fully shaded. Then in particular $C=F$, as only edges contained in the final tree $F$ are shaded during the run of the algorithm. The first lemma connects this criterion to the (UNO) property. %: This is the case if we obtain a component with all nodes well-connected.
%The first lemma states a criterion to detect that the current tree in a run of  already is the fully shaded final tree with all edges shaded: This is the case if we obtain a component with all nodes well-connected.
\begin{lemma}\label{lem: done}
Let $C$ be a partially shaded tree, corresponding to a vertex of a transportation polytope $\TP$.
If there is a well-connected component without an open node in $C$, then it is because $C$ is fully shaded.
\end{lemma}

\begin{proof}

Let $K$ be a component in the current tree $C$ with all nodes well-connected and assume $C$ is not fully shaded already. 
Let $V$ be the set of nodes  of the underlying bipartite graph. Then in particular, $V(K)$ is not connected to $V-V(K)$ by a shaded edge in $C$, but only by an unshaded edge. Similarly, $V(K)$ and $V-V(K)$ are connected in $F$ by at least one edge $e$. This edge $e$ cannot exist in $C$ already: Assume it does. Then $e$ has to be unshaded. Otherwise it would be a shaded edge incident to a well-connected node in $K$ and thus would be a well-connected edge which makes component $K$ larger. There are no unshaded edges incident to well-connected demand nodes, so $e$ has to be incident to a supply node in $V(K)$. This contradicts the definition of a well-connected supply node. Thus $e$ still has to be inserted and it must be incident to a demand node of the component $K$.

As we explain in Section \ref{sec: TP prelim} it is standard practice to think of a transportation problem as a min-cost flow problem on a bipartite network and we will make use of this now. The feasible flows for given margins $u,v$ correspond to the maximum flows in this network. Thus the difference of two feasible flows is a circulation, which is well-known to decompose into a set of cycles with flow through them \cite{ff-62}. 
Therefore, the difference $y^F-y^C$ of the feasible flows $y^C,y^F$ of two trees $C$ and $F$ can be decomposed into such a set of cycles. Even more, there is such a decomposition such that flow on edges with  $\{\si^i,\di^j\}$ with $y^F_{ij}>y^C_{ij}$ is only increased and flow on edges  with $y^F_{ij}<y^C_{ij}$ is only decreased. 
%In particular, for an edge $\{\si^i,\di^j\}$ with $y^F_{ij}>y^C_{ij}$, there is a cycle $\mathcal{C}$ in any such decomposition that increases flow on this edge. 
In particular there is a cycle $\mathcal C$ that increases flow on our edge $e$ that exists in $F$ but not in $C$. Then the cycle $\mathcal{C}$ connects $V(K)$ and $V-V(K)$ by at least one other edge $e'\neq e$. The edge $e'$ is unshaded or does not even exist in the current tree $C$; otherwise it would be well-connected. 

Note that any pair of a supply node and a demand node is connected by a path with an odd number of edges. Every pair of two demand nodes is connected by a path with an even number of edges. Further, flow on the edges in the cycle $\mathcal{C}$ is alternatingly increased and decreased. Thus, flow on the edge $e'$ has to be increased, if it is incident to a supply node, or decreased, if it is incident to a demand node of $K$. But there are no such edges because for the supply nodes of $K$, all edges connecting to $V-V(K)$ are edges to delete (with flow to decrease) and for the demand nodes all such edges are edges to insert (flow to increase), a contradiction. Thus $C$ is fully shaded.  \hfill \qed
\end{proof}

During a run of our algorithm, we will keep the (UNO) property for all trees before the final step. Because of this, the above statement will result in a termination criterion. The next lemma is a simple but useful observation.

\begin{lemma}\label{lem: not all +}
Let $C$ be a partially shaded tree, corresponding to a vertex of a transportation polytope $\TP$.
Assume that in  $C$, there is a supply node satisfying (SIN). Then there is no demand node only 
incident to shaded $+$edges in $C$. %, unless the current tree equals the final tree. 
In particular, every well-connected demand node is incident to a shaded $-$edge. 
\end{lemma}

\begin{proof}
Assume there is a demand node $\di$ only incident to shaded $+$edges. Then, for any supply node, one of these shaded $+$edges is an odd edge. 
But this is a contradiction to having a supply node with the (SIN) property in $C$.  \hfill \qed
\end{proof}

\subsection{Proof of Theorem \ref{main thm}}

Before starting with the actual proof, let us briefly stress its main points. The key point is to show 
that we always avoid deletion of a shaded edge. We do so in the following way (which precisely 
corresponds to Algorithms \ref{algo:hirschwalk} and \ref{algo:newsupply}):

\begin{itemize}
\item If the current tree satisfies (UNO), then deletion of a shaded edge can be avoided if we (insert and) shade an edge incident to a supply node satisfying (SIN); first all $-$edges are shaded and only then the unique $+$edge. (this is the content of Lemma \ref{lem: delete unshaded}).
\item When proceeding like this, we keep the (UNO) property throughout the whole walk from $O$ to $F$. Further we can always find a new supply node with the (SIN) property for the next iteration. (this is the content of Lemma \ref{lem: insert edge}).
\end{itemize}

In the end, we combine our results to obtain Theorem \ref{main thm}.

We must also stress again that the margins of a (non-degenerate) transportation polytope determine which trees appear as vertices of a transportation polytope 
and which unique edge is deleted when inserting an edge into a tree.  Therefore, we can avoid dealing with explicit margins in our proofs. Instead our arguments 
always refer to the unique edge that is deleted, which allows for a less technical proof. Note that the edge is unique because of our assumption on the polytope being non-degenerate.

%Recall that the margins of a transportation polytope determine which trees appear as vertices and which edge is deleted when inserting an edge into a tree. An important aspect of the proof is that we can avoid dealing with explicit margins by performing only very few case distinctions based on which edges are deleted. This allows for a less technical proof.

%Recall that Algorithm \ref{algo:hirschwalk} has to stop if the current tree $C$ equals the final tree $C$ with all edges shaded. But this is the case if and only if all edges in $C$ are shaded (\emph{`$C$ is fully shaded'}). 

%\newpage

In Lemma \ref{lem: delete unshaded} we prove the key aspect for the correctness of our algorithm: The properties (UNO) and (SIN) and our shading order 
(first $-$edges, then the unique $+$edge) imply that no shaded edge is ever deleted. 
%Note that the order for shading the edges in assumption 3 in Lemma \ref{lem: delete unshaded} is precisely the order as in Algorithm \ref{algo:hirschwalk}. 

%In Lemma \ref{lem: insert edge} we will see that Algorithms \ref{algo:hirschwalk} and \ref{algo:newsupply} indeed follow those principles.

%Recall that the key claim for correctness of our algorithm is that we never delete an edge that was (inserted and) shaded at some point. Indeed, Lemma \ref{lem: delete unshaded} proves that if (UNO) holds for the current tree, then there are edges that can be inserted such that no shaded edge is deleted. In Lemma \ref{lem: insert edge} we will see that Algorithm \ref{algo:hirschwalk} in fact performs such an insertion.

%{\bf MY FIXING OF THIS LEMMA}

\begin{lemma}\label{lem: delete unshaded}
Let $C$ be a partially shaded tree, corresponding to a vertex of a transportation polytope $\TP$, and
assume the following three conditions hold for $C$:
\begin{enumerate}
	\item (UNO) holds in $C$. 
	\item There are no shaded $+$edges incident to open supply nodes.
	\item There is a supply node $\sigma$ satisfying (SIN) in $C$. 
\end{enumerate}

\noindent Choose an  edge $\se$ incident to $\sigma$, as follows:
\begin{itemize}

\item[ ] If  in the current tree $C$ there is a $-$edge incident to $\sigma$ still left to be (inserted and) shaded, then take edge $e$ to be one such edge. 
\\Otherwise, $e$ is picked to be the unique $+$edge incident to $\sigma$.
\end{itemize}
\noindent Under this rule of selection, no shaded edge is deleted when (inserting and) shading $\se$ in $C$ to produce $\suc$.
\end{lemma}

\begin{proof}
If we only shade an existing edge, the statement is obvious because in that operation we delete no edges.
%It suffices to consider the case where an edge is inserted. 
%So assume the step from tree $C$ to tree $\suc$ inserts $\se=\{\si,\di\}$ and deletes $\{\si',\di'\}$. 
Let  $\se=\{\si,\di\}$ be the edge we insert into the tree $C$. Recall that the margins of $\TP$ 
% determine the flow on the edges of the tree $C$, and thus in particular
 determine the edge $\{\si',\di'\}$ that is deleted in this step. Let $\suc$ be the succeeding tree. 
Now assume $\{\si',\di'\}$ was shaded in $C$. Note that by (SIN) for $\si$ in $C$, $\{\si',\di'\}$ was a $-$edge with $\di'$ well-connected in $C$.  
Let $K$ be the component in $C$ that contains both $\si'$ and $\di'$. Then $E(K)-\{\si',\di'\}$ induces two connected components, the component 
$K_{\di'}$ containing $\di'$ and $K_{\si'}$ containing $\si'$.

First observe that the open node of component $K$ (in $C$) must be contained in $K_{\si'}$. To see this, assume all 
nodes in $K_{\si'}$ are well-connected. Then each supply node must be incident to a shaded $+$edge and by 
Lemma \ref{lem: not all +} each demand node is incident to a shaded $-$edge. Further, all these edges are contained 
in $K_{\si'}$. But then $K_{\si'}$ has at least $|V(K_{\si'})|$ many edges, a contradiction to being cycle-free as a subgraph of a tree.

Thus, all nodes in the other connected component $K_{\di'}$ are well-connected in $C$ and therefore also in $\suc$. This is because no unshaded edge is inserted (in particular not incident to a demand node in $K_{\di'}$) and in this step no edge incident to a supply node in $K_{\di'}$  is deleted.
%This is because by definition, inserting and shading an edge cannot make a node open again.
 %, as we do not delete a shaded edge incident to a supply node in $K_{\di'}$ and further we cannot get a new unshaded edge incident to a demand node.
If $\{\si,\di\}$ does not connect to $K_{\di'}$, then $K_{\di'}$ forms a component in $\suc$ and all nodes in this component are well-connected. Lemma \ref{lem: done} implies $\suc=F$, a contradiction to $\{\si',\di'\}\in F\backslash \suc$. 

\begin{figure}[H]
\center
\begin{tikzpicture}[scale= 0.6]

\coordinate (v4) at (-1,0)  ;
\coordinate (v1) at (0,2)  ;
\coordinate (v2) at (0,0)  ;
\coordinate (v3) at (1,0)  ;
\coordinate (v9) at (3,2)  ;
\coordinate (v8) at (2.7,0)  ;
\coordinate (v10) at (4,0)  ;
\coordinate (v11) at (5,2)  ;
\coordinate (v12) at (5,-1.5)  ;
\coordinate (v15) at (5.5,-3)  ;
\coordinate (v14) at (4.5,-3)  ;
\coordinate (v13) at (6,0)  ;
\coordinate (v5) at (-2,2)  ;
\coordinate (v6) at (-3,0)  ;

\draw[line width= 1.5]  (v1) -- (v2)node[midway] {$ \;\;+$};
\draw[line width= 1.5]  (v1) -- (v3)node[midway] {$\;\; +$};
\draw[line width= 1.5]  (v1) -- (v4)node[midway] {$ -\;\;\;$};
\draw [line width= 1.5]  (v4) -- (v5) node[midway] {$ +\;\;\;$} ;
\draw[line width= 1.5]  (v6) -- (v5)node[midway] {$-\;\;\; $} ;
\draw[line width= 1.5]  (v8) -- (v9) node[midway] {$ \;\;+$};
\draw[dotted]  (v9) -- (v3) node[midway] {$-$};
\draw[line width= 1.5]  (v9) -- (v10) node[midway] {$ \;\;+$};
\draw[line width= 1.5]  (v11) -- (v10) node[midway] {$- \;\;$};
\draw[line width= 1.5, dashed]  (v11) -- (v12) node[midway] {$ \;\;+$};
\draw[line width= 1.5]  (v11) -- (v13)node[midway] {$ \;\;\; +$};

%\draw [fill, black] (v3) circle (0.1cm);
%\draw [fill, black] (v9) circle (0.1cm);
\node[above] at (v11) {$\di$};
%\draw [fill, black] (v11) circle (0.1cm);
\node[above] at (v9) {$\di'$};
\node[below] at (v3) { $\si'$};
%\draw [fill, black] (v12) circle (0.1cm);
\node[right] at (v12) {$\si$};
\draw[line width = 1.5] (v12) -- (v15) node[midway] {$ \;\;-$};
\draw[line width = 1.5] (v12) -- (v14) node[midway] {$ -\;\;$};

\draw[line width = 0.05]  (4.4, 1.3) ellipse (2.4 and 2);
\draw[line width = 0.05]  (-1, 0.8) ellipse (2.4 and 2);
\draw[line width = 0.05]  (5, -2.3) ellipse (1.2 and 1.2);
\node at (-1,2.8) {$K_{\si'}$};
\node at (6.8,1.7) {$K_{\di'}$};
\end{tikzpicture}\caption{The dotted edge $\{\si',\di'\}$ was deleted and the bold dashed edge $\{\si,\di\}$ was inserted; ovals illustrate $K_{\si'}$, $K_{\di'}$, and the component with open node $\si$ in $C$. In this example, the unnamed left-most supply node is the open node in $K_{\si'}$. }\label{fig:sin}
\end{figure}

Therefore, $\{\si,\di\}$ has to connect to $K_{\di'}$, so one of $\si$ or $\di$ is contained in $K_{\di'}$. All supply nodes in $K_{\di'}$ were already well-connected in $C$, but $\si$ was not. Thus $\di$ must be the node that is contained in $K_{\di'}$. Figure \ref{fig:sin} depicts the situation. Observe that $\di$ must be incident to a $-$edge in $C$ by Lemma \ref{lem: not all +}. Further, recall that there is at most one $-$edge incident to each demand node. Therefore, $\{\si,\di\}$ is a $+$edge. In particular, this is the unique $+$edge incident to $\si$ and thus $\si$ is well-connected in $\suc$ (note that the shaded $-$edge that we deleted was not incident to $\sigma$ as then $F$ would contain a cycle). Therefore, $\{\si,\di\}$ is well-connected in $\suc$ and its insertion merged the components of $\si$ and $\di$ to a larger  component in $\suc$. But this component does not have an open node: We already saw that all nodes in $\delta$'s component $K_{\delta'}$ are well-connected in $\suc$. Further $\si$ was the unique open node of its component, but $\si$ is well-connected now. % all shaded -edges incident to si were well-connected in C by SIN
Again Lemma \ref{lem: done} implies $\suc=F$, but we have $\{\si',\di'\}\in F \backslash \suc$.

Therefore, when inserting $\{\si,\di\}$, we do not delete a shaded edge incident to a well-connected demand node. This proves the claim.  \hfill \qed

\end{proof}

%Note that the order for shading the edges in assumption 3 in Lemma \ref{lem: delete unshaded} is precisely the order as in Algorithm \ref{algo:hirschwalk}. 

Note that, in particular, (UNO) and (SIN) appear as prerequisites for the application of Lemma \ref{lem: delete unshaded}. They appear as conditions 1 and 3.

Lemma \ref{lem: insert edge} is the final statement we need to see the correctness of Algorithm \ref{algo:hirschwalk}. 
We prove that conditions 1 to 3 in Lemma \ref{lem: delete unshaded} remain valid when following Algorithm \ref{algo:hirschwalk}. 
In particular, (UNO) holds in each iteration and Algorithm \ref{algo:newsupply} returns a node satisfying (SIN) in the succeeding tree. 
Further, we prove that the termination criterion is correct. 
The observations from Example \ref{ex: compchanges} are particularly helpful in the proof.

%It remains to show that each iteration of Algorithm \ref{algo:hirschwalk} preserves property (UNO) and continues with a supply node satisfying (SIN). Further, we have to prove that our termination criterion is correct. 

\begin{lemma}\label{lem: insert edge}
Let $C$ be a partially shaded tree, corresponding to a vertex of a transportation polytope $\TP$, and
assume the following three conditions hold for $C$:
\begin{enumerate}
	\item (UNO) holds in $C$. 
	\item There are no shaded $+$edges incident to open supply nodes.
	\item There is a supply node $\sigma$ satisfying (SIN) in $C$. 
\end{enumerate}

\noindent Choose an  edge $\se$ incident to $\sigma$, as follows:
\begin{itemize}

\item[ ] If  in the current tree $C$ there is a $-$edge incident to $\sigma$ still left to be (inserted and) shaded, then take edge $e$ to be one such edge. 
\\Otherwise, $e$ is picked to be the unique $+$edge incident to $\sigma$.
\end{itemize}

\noindent Then, under this rule, one of the following holds:
 \begin{itemize}
 	\item[(i)] Let $\delta^*$ be the demand node only incident to $+$edges in $F$. If $\delta^*$ is well-connected in the succeeding tree $\suc$, then  $\suc$ is fully shaded, that is, it equals the final tree $F$ and all edges are shaded. %\textbf{Better like this?} %($\delta^*$ is the demand node defining the labeling for Algorithm \ref{algo:hirschwalk}.)
 	\item[(ii)] Otherwise, there is still an open supply node with the (SIN) property in the succeeding tree $\suc$ and we can find such a node by proceeding as in Algorithm \ref{algo:newsupply}. In particular, $\suc$ is not fully shaded. Further, (UNO) holds in the tree $\suc$. 
 \end{itemize}

%Then $\delta^*$ became well-connected in this consecutive tree if and only if we arrived at the final tree with all edges shaded. Otherwise (UNO) holds again and we can find an open supply node with (SIN) in the consecutive tree by proceeding as in Algorithm \ref{algo:newsupply}.
\end{lemma}

\begin{proof}
Our proof goes as follows: We will show that if $\suc$ is not fully shaded, then there is a supply node satisfying (SIN) in $\suc$ (see cases 1 and 2). But if there is a supply node with the (SIN) property, then $\delta^*$ cannot be well-connected in $C$ by Lemma \ref{lem: not all +}. Thus, if $\delta^*$ is well-connected, then $\suc$ is fully shaded. This proves (i). 

It remains to prove (ii). Clearly, if $\delta^*$ (or any other node) is not well-connected, then $\suc$ cannot be fully shaded. We show the remaining statements from (ii) in cases 1 and 2. Observe that the selection of a new supply node with the (SIN) property in $\suc$ in cases 1 and 2 matches precisely the routine described in Algorithm \ref{algo:newsupply}. 

We first summarize some fundamental observations.
Let $\se=\{\si,\di\}$ be the edge we (insert and) shade in the tree $C$, yielding the succeeding tree $\suc$. %Before showing that (UNO) is satisfied in $\suc$ and that we can find a supply node with (SIN), we summarize some fundamental observations.
Recall that if we perform an actual insertion, then the margins determine which edge is deleted, and thus they determine the succeeding tree $\suc$.

All nodes/edges that were well-connected in $C$ are still well-connected in $\suc$, as by 
Lemma \ref{lem: delete unshaded} no shaded edge is deleted.
At most two nodes might become well-connected in this step: The supply node $\si$ (which is 
the case if and only if $\{\si,\di\}$ is a $+$edge) and the demand node $\di$ if we only shade an 
edge, respectively, the node $\di'$ incident to the edge deleted in case we perform an insertion.

All shaded edges incident to $\si$ in $C$ are already well-connected by (SIN) for $\si$ in $C$.
In particular, they all belong to the component with open node $\si$.
Therefore, the only edges that might become well-connected are $\{\si,\di\}$ and the shaded edges 
incident to the demand node $\di$ (if only shading), respectively, $\di'$ (if insertion/deletion). 
If we get new well-connected edges, we join several smaller components to a larger component.
\\\\\\
\emph{Case 1} The edge $\{\si,\di\}$ already exists in $C$, so we only shade it. %\smallskip
\\\\
In $C$, $\si$ and $\di$ are the unique open nodes of their components, because there is the edge $\{\si,\di\}$ left to shade at $\si$ and $\di$ is incident to an unshaded edge $\{\si,\di\}$.  %$\di$ is the only demand node that might become well-connected and this is the case if $\{\si,\di\}$ was the only unshaded edge incident to $\di$ in $C$. 
Every supply node $\si'$ connected to $\di$ by an unshaded edge or a $-$edge in $\suc$ is open in $\suc$, as $\si'$ cannot be incident to a shaded $+$edge by the (SIN) property for $\si$ in $C$ (see also Figure \ref{fig:lem4}). 

First assume $\di$ is open in $\suc$. This case is depicted in the left-most picture in Figure \ref{fig:lem4}. Then $\di$ is still incident to an unshaded edge $\{\si',\di\}$ and $\si'$ is open. Further, $\si'$ satisfies (SIN) in $\suc$. 
To see this, note that for all edges, except for $\{\si,\di\}$ and $\{\si',\di\}$, this follows from (SIN) for $\si$ in $C$ (these edges have the same 'parity' with respect to both nodes $\si$ and $\si'$). $\{\si',\di\}$ is the only remaining odd edge for $\si'$. But this edge is unshaded. \\
%To see this, note that for all edges not lying on the path connecting $\si$ and $\si'$ this follows from (SIN) for $\si$ in $C$ (these edges have the same 'parity' with respect to both nodes $\si$ and $\si'$). The path connecting $\si$ and $\si'$ consists of exactly two edges of which $\{\si',\di\}$ is the only odd edge for $\si'$. But this edge is unshaded. \\
For (UNO), observe that the shading affects the components only if $\{\si,\di\}$ is a $+$edge, in which case $\si$ becomes well-connected. But then we joined the two components with open node $\si$ and $\di$, respectively, to a larger component with unique open node $\di$. Thus (UNO) holds in $\suc$. 
\\\\
Otherwise, shading the edge $\{\si,\di\}$ makes $\di$ well-connected. If $\di$ is only incident to shaded $+$edges in $\suc$, then $\suc$ is fully shaded. The reason is that $\di$ is well-connected, and thus the edge $\{\si,\di\}$ also is well-connected. Therefore, we joined the two components with open nodes $\si$ and $\di$ to a single component in $\suc$. But this larger component does not have an open node, as both $\si$ and $\di$ are well-connected in $\suc$ (recall that $\si$ is incident to the shaded  $+$edge $\{\si,\di\}$). Therefore, we reached $F$ with all edges shaded by Lemma \ref{lem: done}.

Thus, if $\suc$ is not fully shaded, then $\di$ is incident to its unique $-$edge $\{\si'',\di\}$ in $\suc$ and it is shaded (note that $\si''=\si$ is possible); see center and right-hand side of Figure \ref{fig:lem4}. Recall that $\si''$ is open in $\suc$ and observe that it satisfies (SIN) in $\suc$. This is because the odd edge $\{\si'',\di\}$ is a $-$edge incident to a well-connected demand node. For the remaining odd edges the property follows from (SIN) for $\si$ in $C$. 

For (UNO), observe that  $\{\si,\di\}$ and $\{\si'',\di\}$ became well-connected. $\si$, $\si''$, and $\di$ were  open in $C$, but $\di$ is well-connected in $\suc$. Therefore, we merged the components with open nodes $\si$, $\si''$, and $\di$, respectively, to a larger component. 
Further, $\si''$ is the unique open node of this component. For $\si=\si''$, this is obvious. Otherwise $\si$ is well-connected in $\suc$, as $\{\si,\di\}$ has to be a $+$edge. This is because there is at most one $-$edge incident to every demand node, which is $\{\si'',\di\}$ in this situation. Therefore, (UNO) holds in $\suc$. This concludes case $1$.

\begin{figure}
\center
\begin{tikzpicture}[scale= 0.6]

\coordinate (v4) at (-1,0)  ;
\coordinate (v1) at (0.4,2)  ;
\coordinate (v2) at (2,0)  ;
\coordinate (v3) at (0.5,0)  ;
\coordinate (v5) at (-1.75,-2)  ;
\coordinate (v6) at (-0.25,-2)  ;

\draw[line width= 1.5]   (v4) -- (v5)node[midway] {$ -\;\;\;$};
\draw[line width= 1.5]   (v4) -- (v6)node[midway] {$\;\;\; -$};
\draw[line width= 1.5]   (v1) -- (v3)node[midway] {$\;\;\; +$} ;
\draw[line width= 1.5, dashed]  (v1) -- (v4);
\draw  (v1) -- (v4);
\draw (v1) -- (v2);
%\draw [fill, black] (v4) circle (0.1cm);
\node[left] at (v4) {$\si$};
%\draw [fill, black] (v1) circle (0.1cm);
\node[above] at (v1) {$\di$};
\node[right] at (v2) {$\si'$};

\draw[line width = 0.05]   (-1.3,-1.4) ellipse (1.4 and 1.7);
\draw[line width = 0.05]   (0.45,1.2) ellipse (0.6 and 1.8);
\end{tikzpicture}
\hspace*{1cm}
\begin{tikzpicture}[scale= 0.6]
\coordinate (v4) at (-1,0)  ;
\coordinate (v1) at (0.4,2)  ;
\coordinate (v2) at (2,0)  ;
\coordinate (v3) at (0.5,0)  ;
\coordinate (v5) at (-1.75,-2)  ;
\coordinate (v6) at (-0.25,-2)  ;

\draw[line width= 1.5]   (v4) -- (v5)node[midway] {$ -\;\;\;$};
\draw[line width= 1.5]   (v4) -- (v6)node[midway] {$\;\;\; -$};
\draw[line width= 1.5]   (v1) -- (v3)node[midway] {$\;\;\; +$} ;
\draw[line width= 1.5, dashed]  (v1) -- (v4);
\draw  (v1) -- (v4);
\draw[line width= 1.5] (v1) -- (v2)node[midway] {$\;\;\;\; -$};
%\draw [fill, black] (v4) circle (0.1cm);
\node[left] at (v4) {$\si$};
%\draw [fill, black] (v1) circle (0.1cm);
\node[above] at (v1) {$\di$};
\node[right] at (v2) {$\si''$};

\draw[line width = 0.05]   (-1.3,-1.4) ellipse (1.4 and 1.7);
\draw[line width = 0.05]   (0.45,1.2) ellipse (0.6 and 1.8);
\draw[line width = 0.05]   (2.5, -0.2) ellipse (1 and 0.8);
\end{tikzpicture}
\hspace*{1cm}
\begin{tikzpicture}[scale= 0.6]
\coordinate (v4) at (-1,0)  ;
\coordinate (v1) at (0.4,2)  ;
\coordinate (v3) at (0.5,0)  ;
\coordinate (v5) at (-1.75,-2)  ;
\coordinate (v6) at (-0.25,-2)  ;

\draw[line width= 1.5]   (v4) -- (v5)node[midway] {$ -\;\;\;$};
\draw[line width= 1.5]   (v4) -- (v6)node[midway] {$\;\;\; -$};
\draw[line width= 1.5]   (v1) -- (v3)node[midway] {$\;\;\; +$} ;
\draw[line width= 1.5, dashed]  (v1) -- (v4)node[midway] {$-\;\;\; $};
\draw  (v1) -- (v4);
%\draw [fill, black] (v4) circle (0.1cm);
\node[left] at (v4) {$\si''=\si$};
%\draw [fill, black] (v1) circle (0.1cm);
\node[above] at (v1) {$\di$};

\draw[line width = 0.05]   (-1.3,-1.4) ellipse (1.4 and 1.7);
\draw[line width = 0.05]   (0.45,1.2) ellipse (0.6 and 1.8);
\end{tikzpicture}
\caption{Possible configurations when shading $\{\si,\di\}$ (from left to right): $\di$ open, $\di$ well-connected and $\si''\neq \si$, $\di$ well-connected and $\si''= \si$. The circles illustrate the components in $C$.}\label{fig:lem4}
\end{figure}

\noindent\emph{Case 2}  We inserted and shaded $\{\si,\di\}$ and deleted an unshaded edge incident to the demand node $\di'$ (possibly with $\di'=\di$). \\\\
In $C$, $\si$ and $\di'$ are the unique open nodes of their components, because there is the edge $\{\si,\di\}$ left to insert at $\si$ and $\di'$ is incident to an unshaded edge. 
%$\di'$ is the only demand node that might become well-connected. This is the case if and only if the edge we deleted was the only unshaded edge incident to $\di'$ in $C$. \\
Observe that every supply node $\si'$ connected to $\di'$ in $\suc$ by an unshaded edge or a $-$edge is open in $\suc$. To see this, recall that a supply node is open if and only if it is not incident to a shaded $+$edge. Thus we have to show that there is no shaded $+$edge incident to $\si'$. $\{\si',\di'\}$ is no shaded $+$edge by assumption. Any other edge incident to $\si'$ in $C$ is an odd edge with respect to $\si$ and thus cannot be a shaded $+$edge by (SIN) for $\si$ in $C$. Thus $\si'$ is open in $C$ and, if $\si'\neq \si$,  $\si'$ is also open in $\suc$. \\
For the case $\si'=\si$, first observe that if $\{\si,\di'\}\in \suc$, then this is the edge we inserted: As the step from $C$ to $\suc$ deletes an edge incident to $\di'$ and inserts an edge incident to $\si$, $\{\si,\di'\}$ must be contained in the cycle that describes this step. We delete (decrease) another edge incident to $\di'$, so $\{\si,\di'\}$ must be increased in this step. But the increased edge incident to $\si$ is the edge we insert. Thus, if $\{\si,\di'\}\in \suc$, then $\{\si,\di'\}$ has to be the edge we inserted and if it is a shaded $-$edge, then $\si$ is open by our shading order (by condition 4: first $-$edges, then the unique $+$edge). Figure \ref{fig:case2} depicts the situations that may arise in Case 2. 
%For the case $\si'=\si$, observe that $\{\si,\di'\}\notin C\cap \suc$. To see this, assume $\{\si,\di'\}\in C\cap \suc$. Then this edge is contained in the cycle that describes the step from $C$ to $\suc$ and it is neither the inserted edge nor the deleted edge. We delete (decrease) another edge incident to $\di'$, so $\{\si,\di'\}$ must be increased in this step. But the increased edge incident to $\si$ is the edge we insert, a contradiction. Thus, if $\{\si,\di'\}\in \suc$, then $\{\si,\di'\}$ has to be the edge we inserted. So in particular this edge is shaded. Further, if it is a shaded $-$edge, then $\si$ is open by our shading order (by condition 4: first $-$edges, then the unique $+$edge). Figure \ref{fig:case2} depicts the situations that may arise in case 2. 
\\\\
First assume $\di'$ is open in $\suc$. Then $\di'$ is still incident to an unshaded edge $\{\si',\di'\}$ and $\si'$ is open by the above argument. Further,  $\si'$ satisfies (SIN) in $\suc$: The odd edge $\{\si',\di'\}$ is unshaded. The inserted edge $\{\si,\di\}$ is even for $\si'$. For all remaining edges, it follows from (SIN) for $\si$ in $C$ (note that after inserting $\{\si,\di\}$ and deleting an edge incident to $\di'$, odd (respectively even) edges for $\si'$ in $\suc$ are odd (respectively even) with respect to $\si$ in $C$; see also Figure \ref{fig:case2}). 

As $\di'$ stays open in $\suc$, the components are affected only if the inserted edge $\{\si,\di\}$ is well-connected in $\suc$. If $\{\si,\di\}$ is a $-$edge, then $\si$ is still open in $\suc$. In case $\di=\di'$, $\di$ is open by assumption. If $\di\neq\di'$, then $\di$ is open in $\suc$, as $\di$ is incident to an unshaded edge in $C$ by Lemma \ref{lem: not all +} and as $\di \neq \di'$,  $\di$ is still incident to this unshaded edge in $\suc$. Thus, in both cases, the components are not affected if $\{\si,\di\}$ is a $-$edge. 
If $\{\si,\di\}$  is a $+$edge, the step from $C$ to $\suc$ joins the components of $\si$ and $\di$ to a larger component whose unique open node is the one from $\di$'s component. In particular, (UNO) holds in $\suc$. 
\\\\
Now assume that the deletion of the edge incident to $\di'$ makes $\di'$ well-connected. If $\di'$ is only incident to shaded $+$edges in $\suc$, then $\suc$ is fully shaded: 
If we inserted the $+$edge $\{\si,\di'\}$, then we joined the components with open nodes $\si$ and $\di'$, respectively, to a larger component. As $\si$ and $\di'$ are well-connected in $\suc$, this component in $\suc$ does not have an open node. Otherwise, all shaded $+$edges incident to $\di'$ in $\suc$ are already well-connected in $C$. Thus the component with open node $\di'$ in $C$ does not change, but it does not have an open node in $\suc$. In either case, we reached $F$ with all edges shaded by Lemma \ref{lem: done}.
%Now assume that the deletion of the edge incident to $\di'$ makes $\di'$ well-connected. If $\di'$ is only incident to shaded $+$edges in $\suc$, then $\suc$ is fully shaded: Unless we inserted the edge $\{\si,\di'\}$ and it is a $+$edge, all shaded $+$edges incident to $\di'$ in $\suc$ are already well-connected in $C$. Thus the component with open node $\di'$ in $C$ does not change, but it does not have an open node in $\suc$. If we inserted the $+$edge $\{\si,\di'\}$, then we joined the components with open nodes $\si$ and $\di'$, respectively, to a larger component. As $\si$ and $\di'$ are well-connected in $\suc$, this component in $\suc$ does not have an open node. In either case,  we reached $F$ with all edges shaded by Lemma \ref{lem: done}.

Thus, if $\di'$ is well-connected in $\suc$ and $\suc$ is not fully shaded, then $\di'$ is incident to its unique $-$edge $\{\si'',\di'\}$ in $\suc$ and the edge is shaded (note that $\si''=\si$ is possible). Recall that $\si''$ is open in $\suc$ and observe that it satisfies (SIN). This is because the odd edge $\{\si'',\di'\}$ is a $-$edge incident to a well-connected demand node. For all other odd edges the property follows just like for $\si'$ in the case where $\di'$ was open.

\begin{figure}
\center
\begin{tikzpicture}[scale=0.6]
\coordinate (v8) at (0,0)  ;
\coordinate (v6) at (1,2)  ;
\coordinate (v7) at (2,0)  ;
\coordinate (v4) at (3.5,2)  ;
\coordinate (v5) at (3,0)  ;
\coordinate (v3) at (4,0)  ;
\coordinate (v1) at (5,2)  ;
\coordinate (v2) at (6,0)  ;
\coordinate (v9) at (2,-2)  ;
\draw (v1) -- (v2);
%\draw[gray, snake=coil, segment aspect=0] (v1)--(v2);
\draw  (v1) -- (v3);
\draw  (v4) -- (v3);
\draw  (v4) -- (v5);
\draw (v6) -- (v7);
%\draw[gray, snake=coil, segment aspect=0] (v6)--(v7);
\draw  (v6) -- (v8);
\draw[ dashed]  (v8) -- (v9);
\draw  (v9) -- (v2);
\draw[dotted]  (v4) -- (v7);
%\draw [fill, black] (v8) circle (0.1cm);
\node[below left] at (v8) {$\si$};
%\draw [fill, black] (v9) circle (0.1cm);
\node[below] at (v9) {$\di$};
\node[above] at (v4) {$\di'$};
\node[below] at (v3) {$\si^.$};
\node[below] at (v5) {$\si^.$};
\draw[gray, snake=coil, segment aspect=0] (v8)--(v6);
\draw[gray, snake=coil, segment aspect=0] (v7)--(v4);
\draw[gray, snake=coil, segment aspect=0] (v3)--(v1);
\draw[gray, snake=coil, segment aspect=0] (v9)--(v2);
\end{tikzpicture}	\caption{Insertion of $\{\si,\di\}$ deletes the dotted edge incident to $\di'$ (wavy: odd edges for $\si$ in $C$). The two $\si^.$ may refer to $\si'$ or $\si''$.} \label{fig:case2}
\end{figure}

For (UNO), observe that $\di'$ and the incident $-$edge $\{\di',\si''\}$ became well-connected in this step. Therefore, we joined the components with open nodes $\di'$ and $\si''$, respectively, to a component with open node $\si''$. As before, the components containing $\si$ and $\di$, respectively, are joined in case $\{\si,\di\}$ is a $+$edge. Note that $\di=\di'$ is possible; then $\si''$ is the open node of the whole component.  This concludes Case $2$.  \hfill \qed

\end{proof}

\noindent Combining all our results, we obtain
\bigskip

\noindent \textit{\bf\em Proof of Theorem \ref{main thm}.}
We have to show that for any two vertices of a non-degenerate transportation polytope $\TP$, Algorithm \ref{algo:hirschwalk} 
finds a walk on the skeleton of $\TP$ of length at most $M+N-1-\crit$ connecting them, where $\crit$ is the number of critical pairs of $\TP$.

First, recall the conditions used as prerequisites for the application of both Lemma \ref{lem: delete unshaded} and \ref{lem: insert edge}: 
\begin{enumerate}
	\item (UNO) holds in $C$. 
	\item There are no shaded $+$edges incident to open supply nodes.
	\item $\sigma$ satisfies (SIN) in $C$. 
\end{enumerate}	
\noindent
%We refer to these prerequisites as conditions 1 to 3. 

Algorithm \ref{algo:hirschwalk} starts with an original tree $O$ in which all edges are unshaded. (UNO) is satisfied as every node forms its own component with exactly one open node (condition 1). As there are no shaded edges, condition 2 is satisfied trivially. Further, all supply nodes satisfy (SIN) (condition 3). Thus we may assume that we are at a partially shaded tree $C$ and have a supply node $\sigma$ such that all of the above conditions are satisfied. 

Clearly, Algorithm \ref{algo:hirschwalk} was designed precisely to choose  an edge $e$ to (insert and) shade following the rule indicated in Lemmas \ref{lem: delete unshaded} and \ref{lem: insert edge}, namely:
\begin{itemize}
\item[ ] If in the current tree $C$ there is a $-$edge incident to $\sigma$ still left to be (inserted and) shaded, then take edge $e$ to be one such edge. 
\\Otherwise, $e$ is picked to be the unique $+$edge incident to $\sigma$.
\end{itemize}

% (UNO) and $\si$ is an open supply node satisfying (SIN). As in Algorithm \ref{algo:hirschwalk}, we (insert and) shade a $-$edge incident to $\si$ if there is such an edge left to shade, otherwise its unique $+$edge.  
%  {\color{purple} Assume that all previous iteration followed this shading strategy as well and did not delete a shaded edge. \textbf{Do we need this?}}
\noindent
Then, by Lemma \ref{lem: delete unshaded}, this step does not delete a shaded edge. Therefore, the succeeding tree $\suc$ satisfies condition 2.
If $\delta^*$ is well-connected in  $\suc$, then we reached the final tree $F$ with all edges shaded by Lemma \ref{lem: insert edge} (i). Thus, our termination criterion as stated in Algorithm \ref{algo:hirschwalk} is correct. Otherwise, by Lemma \ref{lem: insert edge} (ii), the succeeding tree $\suc$ satisfies (UNO) (condition 1) and contains a supply node satisfying (SIN) (condition 3). Such a supply node $\sigma'$ is found by Algorithm \ref{algo:newsupply}. We continue with another iteration, with $\suc$ and $\sigma'$ satisfying conditions 1 to 3.

We now show that the sequence produced by Algorithm \ref{algo:hirschwalk} has length at most $N_1+N_2-1-\crit$. As we shade every edge we insert and we never delete a shaded edge, every edge is inserted at most once. We only insert edges contained in $F$ and there are exactly $N_1+N_2-1$ such edges. However, the edges corresponding to the $\mu$ critical pairs exist in every tree. Therefore, Algorithm \ref{algo:hirschwalk} does not perform an insertion when shading them. Thus, we have at most $N_1+N_2-1-\crit$ steps along the skeleton of the transportation polytope. 

Finally, it remains to show that the faces of the polytope satisfy the Hirsch conjecture as well. To see this, note that a face of a transportation polytope is described by a set of edges that do not exist in any support graph of its vertices. Therefore, the walk constructed by Algorithm \ref{algo:hirschwalk} stays in the face (of minimum dimension) containing the two vertices corresponding to the trees $O$ and $F$, because the trees of the sequence only contain edges from $O\cup F$. Further, observe that when restricting to a face of a transportation polytope, the dimension of the face and its number of facets both are reduced by the same number $k$, which is the number of edges that are never used and thus the Hirsch bound does not change.

This completes the proof of Theorem \ref{main thm}. The Hirsch conjecture is true for all transportation polytopes and all their faces.
%
%Altogether, Algorithm \ref{algo:hirschwalk} only (inserts and) shades edges that are contained in the final tree and every edge we insert is shaded. Also, we never delete a shaded edge. We do not have to insert the non-critical edges as they exist in every tree; they are shaded without an insertion at some point. Therefore, every non-critical edge contained in $F$ is inserted at most once and there are exactly $M+N-1-\crit$ such edges. Thus, we have a walk on the skeleton of the This proves the claim.
\hfill \qed
\bigskip
%\end{proof}

\noindent  {\bf Remark.\; } It is worth noticing that the walk we obtain from the initial tree $O$ to the final tree $F$ is not always necessarily a shortest path, but there are certainly transportation polytopes for
which this is the case. Moreover, the walk produced by Algorithm \ref{algo:hirschwalk} is not necessarily monotone with respect to a linear objective function, unlike the walks used by the simplex method. 

At the same time, while the  Hirsch conjecture for transportation polytopes is true and guarantees a short walk between vertices, it is known that there are \emph{long} monotone decreasing walks.  More precisely, for any $1/n>\alpha >0$, the objective function $c \cdot x=x_{1,1}+\alpha x_{1,2}+\dots+ \alpha^{N-1} x_{1,N}+\alpha^N x_{2,1}+\dots+\alpha^{N^2-1} x_{N,N}$
has a monotone decreasing sequence of vertices of the $N {\times} N$-Birkhoff polytope of length $cN!$ for a universal constant $c$ (see \cite{pak}).

A computational enumeration of all $N_1{\times}N_2$-transportation polytopes, for small values of $N_1$ and $N_2$, supports the conjecture that all integers between $1$ and $N_1+N_2-1$ are  diameters of some $N_1 \times N_2$ transportation polytope, but we have not verified this is the case. 

Finally, we remind the reader our arguments are only valid for polytopes. It is still possible that the diameter of the (bounded portion) of the $1$-skeleton of an unbounded network-flow polytope is longer than the Hirsch bound and we leave this as an interesting open question.

\section*{Acknowledgments}

The first author gratefully acknowledges support from the Alexander-von-Humboldt Foundation.
The second author is grateful for the support received through NSF grant DMS-1522158.
The second and third author gratefully acknowledge the support from the 
Hausdorff Research Institute for Mathematics (HIM) in Bonn.

\bibliographystyle {plain} 
\bibliography{literature}

\end{document}

%% file: TP_Intro.tex
\section{Introduction}

Every convex polyhedron has an underlying graph or $1$-skeleton (defined 
by the zero- and one-dimensional faces). The distance between two zero-dimensional faces (vertices) of the polyhedron
is the length of a shortest path between them. The \emph{combinatorial diameter} of a polyhedron is the maximum possible 
distance between a pair of its vertices. Motivated by the study of the worst-case performance of the simplex algorithm to solve 
linear optimization problems, researchers have considered the geometric challenge of deciding what is the largest possible 
(combinatorial) graph  diameter of convex polytopes with given number of facets and dimension (see \cite{DeLoeraOptima,KimSantos10}
 for an overview about this problem and its implications in optimization theory).

One of the most famous conjectures associated with the diameter is the \emph{Hirsch conjecture}, stated in 1957 by Warren M{.} 
Hirsch \cite{Dantzig63}. It claimed an upper bound of $f-d$ on the (combinatorial) diameter of any $d$-dimensional polyhedron with 
$f$ facets. It is known to be true for several special classes of polyhedra (see \cite{KimSantos10} for a list), such as $(0,1)$-polytopes 
\cite{Naddef89} or dual transportation polyhedra \cite{Balinski84}, and some of these results extend to lattice polytopes, 
e.g., \cite{DelPiaMinichi15,DezaManoussakisOnn15}.  Today, we know the Hirsch bound does not hold in general, neither for 
unbounded polyhedra nor for bounded polytopes \cite{KleeWalkup67,Santos12}. It remains a famous open problem to prove or 
disprove that the diameter of polyhedra can be exponential in the dimension and the number of facets of the polyhedron.

Network-flow problems are among the simplest and oldest  linear optimization problems, appearing in all introductory linear programming textbooks. 
It is well-known that the simplex method has a simple graph-theoretical meaning for these problems \cite{vanderbei}. The
network structure allows for a well-known translation of the simplex method steps in terms of arc additions (e.g., variables are arcs,  basic feasible 
bases are some of the spanning trees of the network and the pivot entering/leaving variables are arcs sharing a cycle of the network). Network-flow
problems show exponential pivoting behavior \cite{Zadeh1973}, but unlike the 
general simplex method, we now know versions of the (primal) network simplex method run in polynomial time \cite{orlin}. 
The set of possible feasible flows in the network correspond to the points of a  \emph{network-flow polyhedron}.  Here we restrict ourselves to study network-flow polytopes, i.e., where the solution space is bounded. These polytopes satisfy some special properties
(e.g., for integral input data, they have integer vertices too (see \cite{schrijver-book}),  and there are already some polynomial upper bounds on the diameter of network-flow polytopes (see \cite{bdehn-12,ff-62,orlin}). However, the exact bound on the diameter of network-flow polytopes diameter has also remained an open question. In this paper we present a final solution of this question showing that all network-flow polytopes have a diameter no larger than the Hirsch bound, and that this
is tight in many cases.

%The basis is represented as a rooted spanning tree of the underlying network, in which variables are represented by arcs, and the simplex multipliers by node potentials. At each iteration, an entering variable is selected by some pricing strategy, based on the dual multipliers (node potentials), and forms a cycle with the arcs of the tree. The leaving variable is the arc of the cycle with the least augmenting flow. The substitution of entering for leaving arc, and the reconstruction of the tree is called a pivot. When no non-basic arc remains eligible to enter, the optimal solution has been reached.

\begin{theorem}\label{main-network}
The diameters of all network-flow polytopes satisfy the Hirsch conjecture. There are network-flow polytopes for which the Hirsch bound is tight. As another consequence, the diameter of a network with $n$ nodes and $m$ arcs is no more than the linear bound $m+n-1$.
\end{theorem}

To prove Theorem \ref{main-network} we reduce it to a special type of networks-flow problems and prove the Hirsch conjecture for the corresponding polytopes. A transportation problem models the minimum-cost of transporting goods from  $N_1$ supply nodes to $N_2$ demand nodes, where each of 
these $n=N_1+N_2$ (total) nodes sends, respectively receives, a specified quantity of a product (we assume demand is equal to supply and that the cost is given by the sum  of the costs at each connection). This is readily represented as a network-flow problem on a bipartite network. The network has $N_1\cdot N_2$ arcs.
See \cite{DeLoeraKim13,DeLoeraKimOnnSantos09,KleeWitzgall68,KYK84} and references therein for detailed information about  transportation polytopes.
 In prior work, the Hirsch bound was shown to hold for all $2{\times}N$- and $3{\times} N$-transportation polytopes \cite{BorgwardtDeLoeraFinholdMiller14}. The so-called partition polytopes \cite{Borgwardt13} and the Birkhoff polytopes \cite{BalinskiRussakoff74}, both classes of $0,1$-transportation polytopes, satisfy even much smaller bounds.  Despite a lot of research (see \cite{DeLoeraKim13} for an overview), the value of the exact diameter of transportation polytopes has remained an open problem until now. In this paper we finally solve this problem.  

For general transportation polytopes, the best published bound until now was  $8(N_1+N_2-2)$ presented in \cite{BrightwellHeuvelStougie06},
a factor of eight away from the bound claimed by the Hirsch conjecture (other improvements were presented, but are unpublished, see summary in \cite{DeLoeraKim13}).
For $N_1{\times}N_2$-transportation polytopes, the Hirsch conjecture states the diameter bound of a given transportation polytope is $N_1+N_2-1-\crit$, where $\crit$ is 
the number of so-called \emph{critical pairs} of a supply and a demand node. These are the variables that are strictly positive in every feasible solution to our 
transportation problem (a value that is not purely combinatorial, but depends on the demands and supply values). 
Here we finally prove the Hirsch conjecture  holds for all $N_1{\times}N_2$-transportation polytopes and their faces.

\begin{theorem}\label{main thm}
The diameter of an $N_1 {\times} N_2$-transportation polytope is bounded above by $N_1+N_2-1-\crit$, 
where $\crit$ is the number of critical pairs of the transportation polytope. Therefore, the Hirsch conjecture is true for all $N_1 {\times} N_2$-transportation polytopes. More strongly, all faces of $N_1 {\times} N_2$-transportation polytopes satisfy the Hirsch conjecture.
\end{theorem}

One can prove that any network-flow polytope is in fact isomorphic to a face of a transportation polytope. 
%This follows because any network-flow polytope can be turned into a capacitated network-flow polytope. The latter then is isomorphic to a uncapacitated bipartite network polytope-flow, which is in fact the face of some transportation polytope.
%%%This follows because any face of a transportation polytope is the polytope of some uncapacitated bipartite network.
%%%At the same time, the polytope of any capacitated network-flow problem is isomorphic to a face of a transportation polytope. 
Therefore, once we prove Theorem \ref{main thm} we can use it to prove Theorem \ref{main-network}. 

The proof of Theorem \ref{main thm} is an algorithm that connects any two vertices of a transportation polytope by a walk on 
the $1$-skeleton that has length at most  $N_1+N_2-1-\crit$.  Since the walk produced by the algorithm stays in the minimal face containing both vertices the rest of Theorem \ref{main thm} follows immediately. 
Note that  the faces of transportation polytopes are transportation problems with some prescribed `missing edges'. 
From the proof it is clear that the diameter of any $N_1 {\times} N_2$-transportation polytope is never more than $N_1+N_2-1$. We remark that for each value of $N_1,N_2$ with $N_1\geq 3$, $N_2\geq 4$, there exist concrete $N_1{\times}N_2$-transportation polytopes that have no critical pairs and attain the bound of $N_1+N_2-1$ \cite{KYK84}.  Therefore, some transportation polytopes are in fact \emph{Hirsch-sharp polytopes} in the sense of \cite{hk-99,hk-98}.

Our paper is structured as follows. In Section \ref{sec: TP prelim}, we recall the necessary background on network flows and transportation 
polytopes and introduce our general notation. Using the language of networks we show that Theorem \ref{main-network} follows from Theorem \ref{main thm}.
 In Section \ref{sec:algo}, we present our algorithm that constructs a walk on the skeleton of the faces of transportation polytopes
adhering to the bounds in Theorem \ref{main thm}. Section \ref{sec:correct} is dedicated to the correctness proof of the algorithm and the proof of Theorem \ref{main thm}.
%Section \ref{sec:cor} provides the proof of Theorem \ref{main cor}.
%{\bf In Section \ref{sec:cor}, we prove Corollary \ref{main cor}. }

\section{Background and How Theorem \ref{main thm} implies Theorem \ref{main-network}}\label{sec: TP prelim}

Recall a {\em network} is a graph with $n$ nodes and $m$ directed edges (or arcs), where each node $v$ has an integer value
specified, the so called {\em excess} of $v$, and each arc has an
assigned positive integer or infinite value called its {\em capacity}. A {\em
feasible flow} is an assignment of non-negative real values to the arcs of the
network so that for any node $v$ the sum of
values in outgoing arcs minus the sum of values in incoming arcs
equals the prescribed excess of the node $v$ and the capacities of the
arcs are not surpassed. 

The set of all feasible flows with given excess vector $b$
and capacity vector $c$ is a convex polyhedron, the  {\em
network-flow polyhedron}, which is defined by the constraints $\Phi_G
x=b$,\,\,\,$ 0 \leq x \leq c$, where $\Phi_G$ denotes the node-arc
incidence matrix of $G$ (a {\em network matrix}). The incidence
matrix $\Phi_G$ has one column per arc and one row per node. 
% Each column of $\Phi_G$ has as many entries as nodes. 
For an arc going
from $i$ to $j$, its corresponding column has zeros everywhere
except at the $i$-th and $j$-th entries. The $j$-th entry, the  
{\em head} of the arrow, receives a $-1$ and the $i$-th entry, 
{\em tail} of the arrow, a $1$. The optimization problem $\min\; r^\intercal x, \ \Phi_Gx =b, \  0 \leq x \leq c$ is
the {\em min-cost flow} problem \cite{vanderbei}. 

In what follows, we assume the network-flow polyhedra we consider are actually \emph{polytopes}, i.e., they are bounded subsets of space.
Boundedness can be easily checked in terms of the network, by testing for directed cycles on the network (see \cite{greenberg}).
Note that boundedness is obvious for all networks with finite capacities in their arcs. It is worth noticing that it was much harder to disprove
the Hirsch conjecture for polytopes (by Santos \cite{Santos12}) than for unbounded polyhedra (by Klee \& Walkup \cite{KleeWalkup67}). 
Here we leave open the possibility that an unbounded network-flow polyhedron violates the Hirsch conjecture.

We continue with some background on transportation polytopes and their $1$-skeleton.
An $N_1{\times}N_2$-transportation problem has $N_1$ supply points and $N_2$ demand points. Each supply point holds a quantity $u_i>0$ and each 
demand point needs a quantity $v_j>0$ of a product. The vectors $\veu=(u_1,\dots,u_{N_1})$ and $\vev=(v_1,\dots,v_{N_2})$ are the \emph{margins} 
for the transportation polytope. The total supply equals the total demand, so formally $\sum_{i=1}^{N_1} u_i = \sum_{j=1}^{N_2} v_j$. Let $y_{ij}\geq 0$ 
denote the flow from supply point $i$ to demand point $j$. Then the set of \emph{feasible flows} $\vey\in\R^{{N_1}{\times} {N_2}}$ can be described as
\[
		\begin{array}{lcrclcl}
		 &  & \sum\limits_{j=1}^{N_2}  y_{ij}    & =& u_i & \quad & i=1,...,{N_1},\\
		                         &      & \sum\limits_{i=1}^{N_1} y_{ij} & =    &  v_j
		                         & \quad & j=1,...,{N_2},\\
		                         &      &          y_{ij}
		                         & \geq  & 0                   & \quad &  i=1,...,N_1,\text{ }j=1,...,{N_2}.\\
		\end{array}
\]
The set of all real solutions of this system of equations and inequalities constitutes the {\em transportation polytope} $\TP$. 
%\bigskip

When discussing an $N_1{\times}{N_2}$-transportation problem, it is common practice to think of the supply and demand points as nodes in the complete bipartite graph $K_{{N_1},{N_2}}$. In the following we denote supply nodes by $\si$ and demand nodes by $\di$. 
%For a convenient notation in our proof and without loss of generality, we assume that we have indexed both the demand nodes and  supply nodes in such a way that $\{s_.,d_.\}$ is the edge at hand. {\color{red} TODO we can describe this better, can we not?} 
For a feasible solution of a transportation problem, we define the \emph{support graph} as the subgraph of $K_{{N_1},{N_2}}$ that contains precisely the edges of non-zero flow, i.e., $y_{ij}>0$. 
%We denote the nodes corresponding to the supply points $\{\supply_1,\ldots,\supply_M\}$ and the nodes corresponding to the demand points $\{\demand_1,\ldots,\demand_N\}$. For every feasible solution $\vey\in \mathbb{R}^{M{\times}N}$ we define the \emph{support graph} $B(\vey)$ as the subgraph of $K_{M,N}$ that contain precisely the edges $\left\{\{\si,\demand_j\}\ :\  y_{ij}>0,i\in\{1,\ldots,M\}, j\in \{1,\ldots,N\} \right\}$ of non-zero flow.
% We use this representation throughout the paper to visualize our methods. 
%With the term \emph{assignment} we refer to the edge set of the support graph $B(\vey)$ of some $\vey\in \mathbb{R}^{M{\times}N}$. Note that $B(\vey)$ is directly derived from $\vey$.  

In general, the points of transportation polytopes do not have connected support graphs. However, this is the case for the $0$-dimensional faces, or  \emph{vertices}, of non-degenerate transportation polytopes. An $N_1{\times}N_2$-transportation polytope is \emph{non-degenerate} if every vertex of the polytope has exactly $N_1+{N_2}-1$ non-negative entries. It is well-known that this is the 
case if and only if there are no non-empty proper subsets  $I\subsetneq \{1,\ldots,{N_1}\}$ and $J\subsetneq \{1,\ldots,N_2\}$ such that $\sum_{i\in I} u_i=\sum_{j\in J} v_j$, see \cite{KYK84}. Note that for each degenerate $N_1{\times}N_2$-transportation polytope there is a non-degenerate $N_1{\times}N_2$-transportation polytope of the same or larger combinatorial diameter \cite{KYK84}. Therefore, it suffices to consider non-degenerate transportation polytopes to prove upper bounds. We exploit this in the upcoming proof. % \ref{sec:combdiam}.

For transportation polytopes the vertices can be characterized in terms of their support graphs: A feasible solution $\vey$ is a vertex if and only if its support graph contains no cycles, that is, it is a spanning forest. Observe that a vertex is uniquely determined by (the edge set of) its support graph  and the vertices of non-degenerate transportation polytopes are given by spanning trees (see for example \cite{KleeWitzgall68}). Therefore, we refer to the support graphs of vertices of transportation polytopes, as well as to the vertices themselves, simply as \emph{trees} and typically denote them by the capital letters $O$ (for `original' tree), $F$ (for `final tree'), and $C$ and $\suc$ (for the `current' and `succeeding' tree, corresponding to neighboring vertices of the transportation polytope). Here, the margins of the polytope play an important role: They define which trees appear as vertices of the polytope, and which do not.

We continue with characterizing the  $1$-dimensional faces of the transportation polytope in terms of the support graphs. Note that, to avoid confusion, 
we use the term `edge' only  for the edges of the underlying bipartite graphs, but not for the $1$-faces of the transportation polytope.

\begin{proposition}[see e.g., Lemma 4.1 in \cite{KYK84}]\label{prop2}
Let $C$ and $\suc$ be two trees that correspond to vertices of an $N_1{\times}N_2$-transportation polytope $\TP$. 
Then the vertices are adjacent in the $1$-skeleton of $\TP$ if and only if $C\cup \suc$ contains a unique cycle.
%Then the vertices are connected by an edge if and only if $C\cup \suc$ contains a unique cycle.
\end{proposition}

In particular, walking from some vertex to a neighboring vertex in the polytope (taking a step on the skeleton / walking along a $1$-face of the polytope) 
corresponds to changing the flow on the edges of $K_{N_1,N_2}$:  Being at a vertex (spanning tree) $C$, we insert an arbitrary edge $\{\si,\di \}\notin C$ into $C$. 
This closes a cycle of even length. We alternately increase and decrease flow on the edges of this cycle, where we increase on the edge we inserted. All edges 
are changed by the same amount which is the minimum existing flow among the edges that are decreased. Due to non-degeneracy (which we may assume here) 
this deletes exactly one edge and hence leads to a tree $\suc$ that is a neighboring vertex of the transportation polytope. Note that the flow on the edges in $C$ is 
determined by the margins of the polytope, so the margins determine which edge is deleted when inserting an edge.

%This closes an even cycle that we orient by using  $(\si,\di)$ instead of $\{\si,\di\}$ for the inserted edge and continuing with the other edges of the cycle to obtain a directed cycle. We change flow in this directed cycle, that is we increase flow on the edges $\{\si,\di\}$ that were directed to be of type $(\si,\di)$ and we decrease flow on the ones of type $(d_j,s_j)$. The value of the flow for this alternating increase and decrease on the edges of the cycles is precisely the minimum existing flow among the edges that are decreased. Due to non-degeneracy (which we may assume here) this deletes exactly one edge and hence leads to a tree $C$ that is a neighboring vertex. 

%Observe that every step inserts one edge and deletes one edge and hence the corresponding support graphs always remain cycle free.
\bigskip

As mentioned before, to prove validity of the Hirsch conjecture, we have to show an upper bound of $f-d=N_1+N_2-1-\crit$ on the combinatorial diameter of $N_1{\times}N_2$ 
transportation polytopes for $\crit$ the number of \emph{critical pairs}. A critical pair $(\si,\di)$ corresponds to an edge $\{\si,\di\}$ that is present in the support graph 
of all feasible solutions. Equivalently, they are the variables satisfying $y_{ij}>0$ for all feasible solutions $\vey$. 
To see that $N_1+N_2-1-\crit$ is in fact the Hirsch bound, note that the dimension of an $N_1{\times}N_2$-transportation polytope is $d=(N_1-1)(N_2-1)$ \cite{KleeWitzgall68}, 
and for non-degenerate transportation polytopes the number of facets is  $f=N_1\cdot N_2 - \crit$. This follows immediately from Theorem 2 in \cite{KleeWitzgall68}. 

\begin{example}[Critical pairs]\label{ex:critical edges}
Consider a $2{\times}3$-transportation polytope with supply nodes $\sigma^1,\sigma^2$ with margins $u_1=5, u_2=3$ and demand nodes $\delta^1,\delta^2, \delta^3$ with margins $v_1=4, v_2=2,v_3=2$. Then the edge $\{\sigma^1,\delta^1\}$ has to be present in every feasible tree: Demand node $\delta^1$ with demand $v_1=4$ cannot  be connected only to the supply node $\sigma^2$ with total supply $u_2=3$. The Figure \ref{fig:ex critical} illustrates two trees that correspond to vertices of the polytope. Nodes are labeled with the margins, edges with the flow. 

	\begin{figure}[H]
		\centering
%			\begin{tikzpicture}[scale=0.8]
%					\coordinate (c1) at (1,0);
%					\coordinate (c2) at (3,0);
%					\coordinate (x1) at	(0,2);
%					\coordinate (x2) at (2,2);
%					\coordinate (x3) at (4,2);
%					\draw [fill, black] (c1) circle (0.1cm);
%					\draw [fill, black] (c2) circle (0.1cm);
%					\draw [fill, black] (x1) circle (0.1cm);
%					\draw [fill, black] (x2) circle (0.1cm);
%					\draw [fill, black] (x3) circle (0.1cm);
%					\node[below] at (c1) {$5$};
%					\node[below] at (c2) {$3$};
%					\node[above] at (x1) {$4$};
%					\node[above] at (x2) {$2$};
%				  \node[above] at (x3) {$2$};
%				  \draw (c1)--(x2) ;
%				  \draw (c2)--(x2) ;
%					\draw (c1)--(x1);
%				  \draw (c2)--(x3) ;			  
%					\node at (2,-1.5) {tree $O$};
%			\end{tikzpicture}
%\hspace*{0.6cm}
			\begin{tikzpicture}[scale=0.6]
					\coordinate (c1) at (1,0);
					\coordinate (c2) at (3,0);
					\coordinate (x1) at	(0,2);
					\coordinate (x2) at (2,2);
					\coordinate (x3) at (4,2);
					\draw [fill, black] (c1) circle (0.1cm);
					\draw [fill, black] (c2) circle (0.1cm);
					\draw [fill, black] (x1) circle (0.1cm);
					\draw [fill, black] (x2) circle (0.1cm);
					\draw [fill, black] (x3) circle (0.1cm);
					\node[below] at (c1) {$u_1=5$};
					\node[below] at (c2) {$u_2=3$};
					\node[above] at (x1) {$v_1=4$};
					\node[above] at (x2) {$v_2=2$};
				  \node[above] at (x3) {$v_3=2$};
				  \draw  (c1)--(x2)node[midway] {$\ \;1$} ;
				  \draw  (c2)--(x2)node[midway] {$1\ \;$} ;
					\draw (c1)--(x1) node[midway] {$4\ \;$} ;
				  \draw  (c2)--(x3)node[midway] {$\ \;2$} ;
%				  \draw[dashed  ] (c2)--(x1);				  
%					\node at (2,-1.5) {pivot $1$};
			\end{tikzpicture}
\hspace*{2cm}
			\begin{tikzpicture}[scale=0.6]
					\coordinate (c1) at (1,0);
					\coordinate (c2) at (3,0);
					\coordinate (x1) at	(0,2);
					\coordinate (x2) at (2,2);
					\coordinate (x3) at (4,2);
					\draw [fill, black] (c1) circle (0.1cm);
					\draw [fill, black] (c2) circle (0.1cm);
					\draw [fill, black] (x1) circle (0.1cm);
					\draw [fill, black] (x2) circle (0.1cm);
					\draw [fill, black] (x3) circle (0.1cm);
					\node[below] at (c1) {$u_1=5$};
					\node[below] at (c2) {$u_2=3$};
					\node[above] at (x1) {$v_1=4$};
					\node[above] at (x2) {$v_2=2$};
				  \node[above] at (x3) {$v_3=2$};				  
				  \draw  (c1)--(x2)node[near end] {$\ \;2$} ;
				  \draw  (c2)--(x1)node[near end] {$1$} ;
					\draw (c1)--(x1) node[midway] {$3\ \;$} ;
				  \draw  (c2)--(x3)node[midway] {$\ \;2$} ;
%				  \draw[dashed] (c2)--(x1);				  
%					\node at (2,-1.5) {pivot $1$};
			\end{tikzpicture}
			\caption{The edge connecting the nodes with margins $5$ and $4$ is present in both trees.}
			\label{fig:ex critical}
\end{figure}

Note that the two trees differ in exactly one edge, and thus are neighbors as vertices in the $1$-skeleton of the transportation polytope. 
In particular, this example shows that even though the edge corresponding to the critical pair is present in any tree, the flow on this edge 
might change during a walk on the skeleton of the transportation polytope. Therefore, we cannot simply disregard the edge and assume to 
have a transportation polytope without critical pairs. \hfill\qed

\end{example}

Note that $N_1+N_2-1-\crit$ is precisely the maximum number of edges in which two trees $O,F$, within the same transportation polytope, can differ. 
In particular, for proving the Hirsch conjecture for $N_1{\times}N_2$-transportation polytopes it is enough to show that there is a finite sequence 
of steps from the original tree $O$ to the final tree $F$ that inserts the edges in $F \setminus O$ one after another such that no inserted edge is deleted at a later point. 
However, the following  example, first mentioned in \cite{BrightwellHeuvelStougie06}, shows that it might be necessary to first delete an edge that is contained in the final tree $F$, 
and only reinsert it at a later step. % A successful procedure will not consists of insertions, but careful deletions. 

\begin{example}\label{THEex1}
Consider the walk from a tree $O$ to a tree $F$ in Figure \ref{fig: THEex for TPs} on the skeleton of a transportation polytope. All supply 
nodes (bottom row) have supply $3$, all demand nodes (top row) have demand $2$. The edges are labeled with the current flow. The 
dashed edges are the edges we insert in the respective pivot:
	\begin{figure}[H]
		\centering
			\begin{tikzpicture}[scale=0.5]
					\coordinate (c1) at (1,0);
					\coordinate (c2) at (3,0);
					\coordinate (x1) at	(0,2);
					\coordinate (x2) at (2,2);
					\coordinate (x3) at (4,2);
					\draw [fill, black] (c1) circle (0.1cm);
					\draw [fill, black] (c2) circle (0.1cm);
					\draw [fill, black] (x1) circle (0.1cm);
					\draw [fill, black] (x2) circle (0.1cm);
					\draw [fill, black] (x3) circle (0.1cm);
%					\node[below] at (c1) {$3$};
%					\node[below] at (c2) {$3$};
%					\node[above] at (x1) {$2$};
%					\node[above] at (x2) {$2$};
%				  \node[above] at (x3) {$2$};
				  \draw (c1)--(x2) ;
				  \draw (c2)--(x2) ;
					\draw (c1)--(x1);
				  \draw (c2)--(x3) ;			  
					\node at (2,-1.5) {tree $O$};
			\end{tikzpicture}
\hspace*{0.6cm}
			\begin{tikzpicture}[scale=0.5]
					\coordinate (c1) at (1,0);
					\coordinate (c2) at (3,0);
					\coordinate (x1) at	(0,2);
					\coordinate (x2) at (2,2);
					\coordinate (x3) at (4,2);
					\draw [fill, black] (c1) circle (0.1cm);
					\draw [fill, black] (c2) circle (0.1cm);
					\draw [fill, black] (x1) circle (0.1cm);
					\draw [fill, black] (x2) circle (0.1cm);
					\draw [fill, black] (x3) circle (0.1cm);
%					\node[below] at (c1) {$3$};
%					\node[below] at (c2) {$3$};
%					\node[above] at (x1) {$2$};
%					\node[above] at (x2) {$2$};
%				  \node[above] at (x3) {$2$};
				  \draw  (c1)--(x2)node[midway] {$1$};
				  \draw  (c2)--(x2)node[midway] {$1$} ;
					\draw (c1)--(x1) node[midway] {$2$} ;
				  \draw  (c2)--(x3)node[midway] {$2$} ;
				  \draw[dashed  ] (c1)--(x3);				  
					\node at (2,-1.5) {pivot $1$};
			\end{tikzpicture}
\hspace*{0.6cm}
			\begin{tikzpicture}[scale=0.5]
					\coordinate (c1) at (1,0);
					\coordinate (c2) at (3,0);
					\coordinate (x1) at	(0,2);
					\coordinate (x2) at (2,2);
					\coordinate (x3) at (4,2);
					\draw [fill, black] (c1) circle (0.1cm);
					\draw [fill, black] (c2) circle (0.1cm);
					\draw [fill, black] (x1) circle (0.1cm);
					\draw [fill, black] (x2) circle (0.1cm);
					\draw [fill, black] (x3) circle (0.1cm);
%					\node[below] at (c1) {$3$};
%					\node[below] at (c2) {$3$};
%					\node[above] at (x1) {$2$};
%					\node[above] at (x2) {$2$};
%				  \node[above] at (x3) {$2$};
				  \draw  (c1)--(x1)node[midway] {$ 2$};
				  \draw[dashed  ] (c2)--(x1);
				  \draw (c1)--(x3)node[near start] {$ 1$};
				  \draw (c2)--(x2)node[near end] {$2 $};
					\draw  (c2)--(x3)node[midway] {$1 $};
					\node at (2,-1.5) {pivot $2$};
			\end{tikzpicture}
\hspace*{0.6cm}
			\begin{tikzpicture}[scale=0.5]
					\coordinate (c1) at (1,0);
					\coordinate (c2) at (3,0);
					\coordinate (x1) at	(0,2);
					\coordinate (x2) at (2,2);
					\coordinate (x3) at (4,2);
					\draw [fill, black] (c1) circle (0.1cm);
					\draw [fill, black] (c2) circle (0.1cm);
					\draw [fill, black] (x1) circle (0.1cm);
					\draw [fill, black] (x2) circle (0.1cm);
					\draw [fill, black] (x3) circle (0.1cm);
%					\node[below] at (c1) {$3$};
%					\node[below] at (c2) {$3$};
%					\node[above] at (x1) {$2$};
%					\node[above] at (x2) {$2$};
%				  \node[above] at (x3) {$2$};
				  \draw  (c1)--(x1)node[midway] {$1 $};
				  \draw (c1)--(x3)node[near end] {$ 2$};
				  \draw  (c2)--(x1)node[near end] {$1 $};
					\draw  (c2)--(x2)node[near end] {$ 2$};
					\draw[dashed  ] (c1)--(x2);
					\node at (2,-1.5) {pivot $3$};
			\end{tikzpicture}
\hspace*{0.6cm}
			\begin{tikzpicture}[scale=0.5]
					\coordinate (c1) at (1,0);
					\coordinate (c2) at (3,0);
					\coordinate (x1) at	(0,2);
					\coordinate (x2) at (2,2);
					\coordinate (x3) at (4,2);
					\draw [fill, black] (c1) circle (0.1cm);
					\draw [fill, black] (c2) circle (0.1cm);
					\draw [fill, black] (x1) circle (0.1cm);
					\draw [fill, black] (x2) circle (0.1cm);
					\draw [fill, black] (x3) circle (0.1cm);
%					\node[below] at (c1) {$3$};
%					\node[below] at (c2) {$3$};
%					\node[above] at (x1) {$2$};
%					\node[above] at (x2) {$2$};
%				  \node[above] at (x3) {$2$};
				  \draw (c1)--(x3);
				  \draw (c1)--(x2);
				  \draw (c2)--(x1);
					\draw (c2)--(x2);
					\node at (2,-1.5) {tree $F$};
			\end{tikzpicture}
		\caption{A walk from tree $O$ to tree $F$ of length three.} \label{fig: THEex for TPs}
	\end{figure} 

%	{\color{purple} TODO check: Sind das die gleichen Schritte, die unser Algorithmus macht? Ich denke, schon...}

Note that $O$ and $F$ differ by only two edges and that all edges that are not in $O$ are in $F$, and vice versa. No matter which edge we insert in the first step, we have to delete an edge that is contained in $F$. Then we need at least two more steps, as we still have to insert two edges from $F$ and can only insert exactly one edge in a single step. Therefore the walk above is a walk of minimum length between $O$ and $F$.
\hfill\qed
\end{example}

It is important to be aware of the fundamental role that the numerical margins play for a walk on the skeleton of a transportation polytope.  
The margins define which particular spanning trees of $K_{N_1,N_2}$ appear as vertices of the polytope. Let us contrast this with a similar family of polytopes.  For the (graphical) matroid polytopes $P_{\text{Mat}(K_{N_1,N_2})}$ of complete bipartite graphs $K_{N_1,N_2}$, the vertices correspond to \emph{all} the spanning trees of $K_{N_1,N_2}$. Like before, two spanning trees differing in exactly one edge correspond to neighboring vertices of the polytope 
$P_{\text{Mat}(K_{N_1,N_2})}$. For a walk on the skeleton of these polytopes, we can simply add the edges containes in the final tree $F$ by exchanging them with edges in the original tree in any order. We do not care about flow values, any two vertices of $P_{\text{Mat}(K_{N_1,N_2})}$ are connected by a walk of length equal to the number of edges in which the two spanning trees differ. In such a sequence we never have to delete an edge that is contained in the target spanning tree $F$, thus the situation from Example \ref{THEex1} does not occur. Thus proving the Hirsch conjecture for graphical matroid polytopes is a purely combinatorial process, whereas the analysis for transportation polytopes requires the margins. 

The following Lemma \ref{reducetobipartite} allows us to derive Theorem \ref{main-network} from Theorem \ref{main thm}.

\begin{lemma} \label{reducetobipartite}

 Given a network $G$ with $n$ nodes and $m$ arcs, with capacity vector $c$
and excess vector $b$, there is a bipartite uncapacitated network
$\widehat{G}$ with $n+m$ nodes ($N_1=n, N_2=m$), $M=2m$ arcs, and excess vector
$\widehat{b}$ (a linear combination of $b$,$c$) such that the
integral flows in both networks are in bijection and the two corresponding polytopes are isomorphic. 
The network $\widehat{G}$ is obtained from $G$ by replacing each arc by two new arcs and a new node and changing the excesses at each node as illustrated in Figure \ref{fig:bounded to unbounded}.
\end{lemma}

 \begin{figure}[h]
 \tikzset{> = latex}
	\centering
	
	\begin{tikzpicture}[scale=0.6]
		\tikzset{vertex/.style = {shape=circle, draw, minimum size=1.5em}}
		\coordinate[vertex] (v1) at (-2,0) {};
		\coordinate[vertex] (v2) at (2,0) {};
		\node at (v1) {\small $i$};
		\node at (v2) {\small $j$};
		\node at (-2,-1) {$b_i$};
		\node at (2,-1) {$b_j$};
		
		\draw[->] (v1)--(v2) node[midway, above] {$c_{ij}$};						
	\end{tikzpicture}
\hspace*{2cm}	
	\begin{tikzpicture}[scale=0.5]
		\tikzset{vertex/.style = {shape=circle, draw, minimum size=1.5em}}
		\coordinate[vertex] (v1) at (-3,0) {};
		\coordinate[vertex] (v2) at (3,0) {};
		\coordinate[vertex] (v3) at (0,1) {};
		\node at (v1) {\small $i$};
		\node at (v2) {\small $j$};
		\node at (v3) {\small $ij$};
		\node at (-3,-1) {$b_i$};
		\node at (3,-1) {$b_j+c_{ij}$};
		\node at (0,2) {$-c_{ij}$};
		
		\draw[->] (v1)--(v3) node[midway, above] {};						
		\draw[->] (v2)--(v3) node[midway, above] {};						
	\end{tikzpicture}
	\caption{Construction for Lemma \ref{reducetobipartite}.}
	\label{fig:bounded to unbounded}
\end{figure}

\proof  For the network $G$  with capacity vector $c,$ the flows are the
solutions of  \begin{equation}\label{eqn}\Phi_G  x \, = \, b , \quad 0 \leq x \leq c.\end{equation}
There is a clear bijection (a deletion of redundant dependent variables) between the solutions of
system (\ref{eqn}) and the solutions of

$$\left [\begin {array}{cc} {\it \Phi_G}&0\\\noalign{\medskip} I& I \end
{array}\right ]\left [\begin {array}{c} x\\\noalign{\medskip}y
\end {array}\right ] = \left [\begin {array}{c}
b\\\noalign{\medskip}c\end {array}\right ], \quad x, y \, \geq \,
0 .$$

The new enlarged matrix is called the
{\em extended network matrix}.  It is clear that the polytope of the original
network $G$ is isomorphic to the polytope of the extended network matrix.
This auxiliary matrix will be useful to prove our claim. From it we apply an
invertible linear map to the polytope as follows.

To the network $G$, with its set of
nodes $V$ and its set of arcs $E$, we have associated the new network
$\widehat{G}$ . The set of nodes of $\widehat{G}$ is the disjoint
union of the two sets $V$ and $E$ and the network $\widehat{G}$ is
obtained from $G$ by replacing each arc by two new arcs as illustrated
in Figure \ref{fig:bounded to unbounded}: that is to each $e\in E$ is associated
$f_1=(j,e)$ and $f_2=(i,e)$ where $e\in E$ is the common head of both of the new arcs $f_1, f_2$ 
and $i,j\in V$ are the tails. Thus $\widehat{G}$ is a directed graph,
with a set $\widehat{V}$ of $n+m$ nodes and a set $\widehat{E}$ of $2m$ 
arcs. We define a
new excess vector $\hat b\in \R^{n+ m}$. The projection of $\hat
b$ on $\R^n$ has coordinates $\hat b_i=b_i+\sum_{f\in E| {
\text{head}}(f)=i}{ c}(f)$. The projection of $\hat b$ on $\R^m$ is
the negative of the capacity vector $c$.

Now we describe the explicit isomorphism between the two polytopes by describing
a (special) linear map between the matrix representation of the old capacitated network and the new
uncapacitated bipartite network. 

Let $T_G\in \mathbb{R}^{n\times m}$ 
be the matrix with one column per arc and one row per node defined as follows. 
The column corresponding to an arc has just {\em one} nonzero entry: 
the {\em head} of the arc receives a $1$.  Then $\Phi_G+T_G$ is the matrix 
with one column per arc and just the {\em tail} of the arc receives a $1$. All
other entries are $0$. Consider the matrix transformation

$$\left [\begin {array}{cc} {\it I}&{\it
T_G}\\\noalign{\medskip} 0& -I \end {array}\right ] \left [\begin
{array}{cc} {\it \Phi_G}&0\\\noalign{\medskip} I& I \end
{array}\right ] =\left [\begin {array}{cc} {\it \Phi_G+T_G}\; & \;{\it
T_G}\\\noalign{\medskip} -I& -I \end {array}\right ]. $$ 

The right-hand side is equal to the incidence matrix $\Phi_{\widehat G}$ for the new network $\widehat G$. 

The $m$ first columns correspond to new arcs $f_1$, and the last $m$ columns
correspond to new arcs $f_2$. Solutions of
$$\left [\begin {array}{cc} {\it \Phi_G}&0\\\noalign{\medskip} I& I \end
{array}\right ]\left [\begin {array}{c} x\\\noalign{\medskip}y
\end {array}\right ] = \left [\begin {array}{c}
b\\\noalign{\medskip}c\end {array}\right ] $$ are  solutions of
the equation
$$ \Phi_{\widehat G} \left [\begin {array}{c} x\\\noalign{\medskip}y
\end {array}\right ]=\left [\begin {array}{cc} {\it I}&{\it T_G}\\\noalign{\medskip} 0& -I \end
{array}\right ]\left [\begin {array}{c} b\\\noalign{\medskip}c\end
{array}\right ]= \hat b.$$

Thus we obtain a linear unimodular bijection between feasible flows of the capacitated network $G$ and feasible flows of the uncapacitated network $\widehat G$. In particular, this implies that both polytopes have the same $1$-skeleton. Note that the isomorphism we provided implies that, if $x_{ij}$ was the value of the flow on arc $(i,j)$ of the original capacitated network, then the corresponding new flow in the uncapacitated network will give $y_{i (i,j)}=x_{ij}$ and $y_{j,(i,j)}=c_{ij}-x_{ij}$.

%Clearly basic feasible solutions are mapped to basic feasible solutions and it establishes a face-isomorphism. 
\qed

\vskip .3cm

\begin{example} In Figure \ref{fig:capnetwork-to-tp} we show an example of the transformation described in Lemma \ref{reducetobipartite}. We also show how to interpret the resulting uncapacitated network as a transportation problem (in this case missing several edges, which are forced by conditions of type $x_{ij}=0$). \hfill\qed

\begin{figure}[htb]
 \tikzset{> = latex}
	\centering
	\begin{tikzpicture}[scale=1.6]
		\tikzset{vertex/.style = {shape=circle, draw, minimum size=1.5em}}
		\coordinate[vertex] (c1) at (-1,0) {};
		\coordinate[vertex]  (c2) at (0,1) {};
		\coordinate[vertex]  (c3) at (0,-1) {};
		\coordinate[vertex]  (c4) at (1,0) {};
		
		\node at (c1) {a};
		\node at (c2) {b};
		\node at (c3) {c};
		\node at (c4) {d};
		
		\node[left] at (-0.1,0) {$5$};% {\small $(3,5) $};
		\node[right] at (0.1,0){$15$};% {\small $ (2,15)$};
		
		\node at (-1.5,0) {$20$};
		\node at (0,1.4) {$0$};
		\node at (1.5,0) {$-20$};
		\node at (0,-1.4) {$0$};
			
		\draw[->] (c1) --  (c2)node[midway, left] {$10$}; % {\small $(5,10)$} ;
		\draw[->] (c1)--(c3) node[midway, left] {$20$}; % {\small $(4,20)$} ;
		\draw[->] (c2) [bend right = 10] to (c3)  ;
		\draw[->] (c2)--(c4)node[midway, right] {$30$}; % {\small $(6,30)$} ;
		\draw[->] (c3)[bend right = 10] to (c2)  ;
		\draw[->] (c3)--(c4)node[midway, right] {$10$}; % {\small $(1,10)$} ;
	\end{tikzpicture}
\vspace{0.5cm}

	\begin{tikzpicture}[scale=0.8]
		\tikzset{vertex/.style = {shape=circle, draw, minimum size=1.5em}}
		\coordinate[vertex] (c1) at (1.25,0) {};
		\coordinate[vertex]  (c2) at (3.75,0) {};
		\coordinate[vertex]  (c3) at (6.25,0) {};
		\coordinate[vertex]  (c4) at (8.75,0) {};
		
		\node at (c1) {a};
		\node at (c2) {b};
		\node at (c3) {c};
		\node at (c4) {d};
		
		\coordinate[vertex] (x1) at	(0,2) {};
		\coordinate[vertex]  (x2) at (2,2) {};
		\coordinate[vertex]  (x3) at (4,2) {};
		\coordinate[vertex]  (x4) at (6,2) {};
		\coordinate[vertex]  (x5) at (8,2) {};
		\coordinate[vertex]  (x6) at (10,2) {};
		
		\node at (x1) {\small {ab}};
		\node at (x2) {\small {ac}};
		\node at (x3) {\small {bc}};
		\node at (x4) {\small {bd}};
		\node at (x5) {\small {cb}};
		\node at (x6) {\small {cd}};
		
		\node at (1.25,-0.75) {$20$};
		\node at (3.75,-0.75) {$25$};
		\node at (6.25,-0.75) {$25$};
		\node at (8.75,-0.75) {$20$};
		\node at (0,2.75) {$-10$};
		\node at (2,2.75) {$-20$};
		\node at (4,2.75) {$-5$};
		\node at (6,2.75) {$-30$};
		\node at (8,2.75) {$-15$};
		\node at (10,2.75) {$-10$};	
			
		\draw[->] (c1)--(x1);%node[near start] {$5$} ;
		\draw[->] (c1)--(x2);%node[near start] {$4$} ;
		\draw[->] (c2)--(x1);%node[near start] {$0$} ;
		\draw[->] (c2)--(x3);%node[near start] {$3$} ;
		\draw[->] (c2)--(x4);%node[near start] {$6$} ;
		\draw[->] (c2)--(x5);%node[near start, right, below] {$\qquad 0$} ; 
		\draw[->] (c3)--(x2);%node[near start, left, below] {$0\qquad $} ; 
		\draw[->] (c3)--(x3);%node[near start] {$0$} ;
		\draw[->] (c3)--(x5);%node[near start] {$2$} ;
		\draw[->] (c3)--(x6);%node[near start] {$1$} ;
		\draw[->] (c4)--(x4);%node[near start] {$0$} ;
		\draw[->] (c4)--(x6);%node[near start] {$0$} ;
	\end{tikzpicture}
	\vspace{0.5cm}

	\begin{tikzpicture}[scale=0.8]
		\coordinate (c1) at (1.25,0);
		\coordinate (c2) at (3.75,0);
		\coordinate (c3) at (6.25,0);
		\coordinate (c4) at (8.75,0);
		
		\coordinate (x1) at	(0,2);
		\coordinate (x2) at (2,2);
		\coordinate (x3) at (4,2);
		\coordinate (x4) at	(6,2);
		\coordinate (x5) at (8,2);
		\coordinate (x6) at (10,2);
		
		\draw [fill, black] (c1) circle (0.1cm);
		\draw [fill, black] (c2) circle (0.1cm);
		\draw [fill, black] (c3) circle (0.1cm);
		\draw [fill, black] (c4) circle (0.1cm);
		\draw [fill, black] (x1) circle (0.1cm);
		\draw [fill, black] (x2) circle (0.1cm);
		\draw [fill, black] (x3) circle (0.1cm);
		\draw [fill, black] (x4) circle (0.1cm);
		\draw [fill, black] (x5) circle (0.1cm);
		\draw [fill, black] (x6) circle (0.1cm);

		\node[below] at (c1) {$u_1=20$};
		\node[below] at (c2) {$u_2=25$};
		\node[below] at (c3) {$u_3=25$};
		\node[below] at (c4) {$u_4=20$};
		\node[above] at (x1) {$v_1=10$};
		\node[above] at (x2) {$v_2=20$};
		\node[above] at (x3) {$v_3=5$};
		\node[above] at (x4) {$v_4=30$};
		\node[above] at (x5) {$v_5=15$};
		\node[above] at (x6) {$v_6=10$};	
			
		\draw (c1)--(x1);
		\draw (c1)--(x2);
		\draw (c2)--(x1);
		\draw (c2)--(x3);
		\draw (c2)--(x4);
		\draw (c2)--(x5);
		\draw (c3)--(x2);
		\draw (c3)--(x3);
		\draw (c3)--(x5);
		\draw (c3)--(x6);
		\draw (c4)--(x4);
		\draw (c4)--(x6);
	\end{tikzpicture}

\caption{From a capacitated network to an uncapacitated bipartite network to a transportation problem.}
	\label{fig:capnetwork-to-tp}
\end{figure}	
\end{example}

\noindent \textit{\bf\em Proof of Theorem  \ref{main-network}.}  Assume that Theorem \ref{main thm} holds, that is, the Hirsch conjecture holds for all the faces of transportation polytopes. We show that this implies the Hirsch conjecture for network-flow polytopes. 

First note that every uncapacitated network-flow problem (or network-flow problem with some uncapacitated arcs), whose solution space is a polytope, is identical to a \emph{capacitated} network-flow problem: For each arc $(i,j)$ with infinite capacity we replace the infinite capacity by a finite capacity that is \emph{strictly} larger than the maximum value of the variable $x_{ij}$. 
This preserves the same polytope. In particular it preserves the number of inequalities that can be satisfied with equality and thus the number of facets (as the new inequalities are never met with equality). 
Thus we can restrict our discussion to capacitated network-flow polytopes. 

The isomorphism  between a capacitated network-flow problem with $n$ nodes and $m$ arcs and the bipartite uncapacitated network with $N_1=n$ supply nodes and $N_2=m$ demand nodes provided by Lemma \ref{reducetobipartite} shows that the two polytopes have the same facets and dimension, as well as the same diameter.  If one satisfies the Hirsch conjecture the other one does too. Now note that the faces of a transportation polytope are described by a set of edges with flow fixed to zero (or, in other words, a set of edges that do not exist). 
Thus the bipartite uncapacitated network constructed in Lemma \ref{reducetobipartite} is in fact a face of a transportion polytope. Therefore, Theorem \ref{main thm} implies that the Hirsch conjecture holds for capacitated network-flow problems, too. 

To see tightness of the Hirsch bound, recall there exist Hirsch-sharp transportation polytopes \cite{KYK84}. It is also easy to give Hirsch-sharp instances for general networks. 

It remains to certify the diameter bound of  $m+n-1$ to the original network-flow polytope. The number of facets $f$ of a capacitated network with $m$ arcs and  $n$ nodes is bounded above by $2m$, the dimension $d$ of the polytope is $2m-(m+n-1)$ (because the rank of the node-arc incidence matrix is $m+n-1$).
As the diameter of the polytope satisfies the Hirsch bound of $f-d$, it also is bounded above by $m+n-1\geq f-d$. Of course, in many instances, the diameter is much smaller.
\qed